\documentclass[12pt]{amsart}
\usepackage[cp1251]{inputenc}
\usepackage[russian]{babel}
\usepackage{amsmath,amsfonts,amssymb}

\title{Базисность Рисса со скобками для системы Дирака с суммируемым потенциалом} 
\author{А.~М.~Савчук, И.~В.~Садовничая} 
\address{Россия, Москва, МГУ имени М.~В.~Ломоносова, механико--математический факультет, Ленинские горы, д.1} 
\email{artem\_savchuk@mail.ru, ivsad@yandex.ru}
\date{}

\theoremstyle{plain} 
\newtheorem{theorem}{Теорема}
\newtheorem{lemma}{Лемма}
\newtheorem{proposition}{Утверждение}
\theoremstyle{definition} 
\newtheorem{definition}{Определение}
\newtheorem{remark}{Замечание}

\numberwithin{equation}{section}
\numberwithin{theorem}{section} \numberwithin{lemma}{section} \numberwithin{proposition}{section} \numberwithin{definition}{section} \numberwithin{remark}{section}

\newtheorem{THEOREM}{Теорема}

\begin{document}
\begin{abstract}
В работе изучается оператор Дирака $\mathcal L_{P,U}$, порожденный в пространстве $\mathbb H=(L_2[0,\pi])^2$ дифференциальным выражением
\begin{gather*}
\ell_P(\mathbf y)=B\mathbf y'+P\mathbf y,\quad \text{где}\\
 B = \begin{pmatrix} -i & 0 \\ 0 & i
        \end{pmatrix},
        \qquad
    P(x) = \begin{pmatrix} p_1(x) & p_2(x) \\ p_3(x) & p_4(x)
        \end{pmatrix},
        \qquad
     \mathbf y(x)=\begin{pmatrix}y_1(x)\\ y_2(x)\end{pmatrix},
\end{gather*}
и регулярными краевыми условиями
$$
U(\mathbf y)=\begin{pmatrix}u_{11} & u_{12}\\ u_{21} &
u_{22}\end{pmatrix}\begin{pmatrix}y_1(0)\\
y_2(0)\end{pmatrix}+\begin{pmatrix}u_{13} & u_{14}\\ u_{23} &
u_{24}\end{pmatrix}\begin{pmatrix}y_1(\pi)\\
y_2(\pi)\end{pmatrix}=0.
$$
Элементы матрицы $P$ предполагаются суммируемыми на $[0,\pi]$ комплекснозначными функциями. Мы покажем, что оператор $\mathcal L_{P,U}$ имеет дискретный спектр,
состоящий из собственных значений $\{\lambda_n\}_{n\in\mathbb Z}$, причем $\lambda_n=\lambda_n^0+o(1)$ при $|n|\to\infty$, где $\{\lambda_n^0\}_{n\in\mathbb Z}$ ---
спектр оператора $\mathcal L_{0,U}$ с нулевым потенциалом и теми же краевыми условиями. Если краевые условия сильно регулярны, то спектр оператора $\mathcal L_{P,U}$
является асимптотически простым. Мы покажем, что в этом случае система собственных и присоединенных функций оператора $\mathcal L_{P,U}$ образует базис Рисса в
пространстве $\mathbb H$ (при условии нормировки собственных функций). В случае регулярных, но не сильно регулярных краевых условий все собственные значения оператора
$\mathcal L_{0,U}$ двукратны, а собственные значения оператора $\mathcal L_{P,U}$ асимптотически двукратны. В этом случае мы покажем, что система, составленная из
соответствующих двумерных корневых подпространств оператора $\mathcal L_{P,U}$, образует базис Рисса из подпространств (базис Рисса со скобками) в пространстве
$\mathbb H$.
\end{abstract}

\maketitle

\tableofcontents
\section*{Введение.}
Спектральная теория краевых задач общего вида для обыкновенных дифференциальных операторов берет свое начало с работ Г.~Биркгофа \cite{B1, B2} и Я.~Д.~Тамаркина
\cite{Ta1, Ta2, Ta3}. В этих работах были введены понятия регулярных и сильно регулярных краевых условий, было исследовано асимптотическое поведение собственных
значений и собственных функций. Кроме того, были доказаны теоремы о полноте системы собственных и присоединенных функций и исследована поточечная сходимость
спектральных разложений. Исследование свойств безусловной базисности (базисности Рисса) системы корневых векторов для обыкновенных дифференциальных операторов
началось в 60-е годы с работ Н.~Данфорда, В.~П.~Михайлова и Г.~Кесельмана \cite{Du, Mikh, Kes}. Тогда же А.~С.~Маркусом \cite{Mar} и В.~Э.~Кацнельсоном \cite{Kaz} был
предложен абстрактный метод, позволяющий доказывать базисность Рисса для возмущений самосопряженных операторов в гильбертовом пространстве. Этот метод получил
существенное развитие в работах А.~С.~Маркуса и В.~И.~Мацаева (см., например, \cite{MarMats}). По поводу применения этого метода к обыкновенным дифференциальным
операторам следует отметить статьи А.~А.~Шкаликова \cite{Sh79, Sh83, Sh82}. В нашей работе мы также используем этот метод. Изучение спектральных свойств
дифференциальных систем первого порядка
\begin{equation*}
iBY'+v(x)Y, \qquad Y=(y_j(x))_1^d,
\end{equation*}
с постоянной $n\times n$ матрицей
$$
B=\begin{pmatrix}b_1&0&0&\dots&0\\0&b_2&0&\dots&0\\0&0&b_3&\dots&0\\ \hdotsfor{5}\\0&0&0&\dots&b_n\end{pmatrix}
$$
и $n\times n$ матриц--функцией $v(x)$ началось с работы Г.~Биркгофа и Р.~Лангера \cite{BL}. Из недавних работ, посвященных таким системам, отметим работы
М.~М.~Маламуда, Л.~Л.~Оридороги и А.~А.~Лунева \cite{MalOri, ML1}. В них введено понятие слабо регулярных краевых условий (для случая системы Дирака оно эквивалентно
обычной регулярности) и доказаны теоремы о полноте, минимальности и базисности Рисса системы корневых векторов для случая $v\in L_\infty[a,b]$. Свойствам базисности
системы собственных и присоединенных функций системы Дирака посвящена обширная литература. И.~Трушин и М.~Ямомото \cite{TM1, TM2} установили базисность Рисса в случае
$P\in L_2$ и разделенных краевых условий. В серии работ П.~Джакова и Б.~Митягина (см., например, \cite{DM1, DM2}) изучаются спектральные свойства оператора Дирака (в
частности, подробно обсуждается случай периодических, антипериодических и общих регулярных, но не сильно регулярных краевых условий). В \cite{DM3} изучен оператор
Дирака с потенциалом $P\in L_2$ и произвольными регулярными краевыми условиями. Для случая сильно регулярных условий была доказана базисность Рисса, а при отсутствии
сильной регулярности --- базисность Рисса из подпространств. В недавней работе \cite{SavSh14} была доказана базисность Рисса для общего случая суммируемого на
$[0,\pi]$ потенциала $Q$ и сильно регулярных краевых условий. Отметим, что в работе А.~А.~Лунева и М.~М.~Маламуда \cite{ML2} также анонсирован этот результат и метод
его доказательства, отличный от предложенного в \cite{SavSh14}. Необходимо также упомянуть работы различных авторов \cite{AHM, Ba12, Kh10}, в которых читатель может
найти близкие результаты. Заметим еще, что свойства базисности естественным образом обобщаются до результатов о равносходимости (см. по этой теме обзорную статью
\cite{Min} и ссылки в ней). Вопросы о равносходимости для системы корневых функций оператора Штурма--Лиувилля с негладкими потенциалами были исследованы вторым
автором в работах \cite{Sad1, Sad2, Sad3}. Таким образом, результаты этой статьи подготовят базу для доказательства таких теорем в случае системы Дирака.

Настоящая статья организована следующим образом. В первом параграфе приведены предварительные результаты, необходимые для дальнешего. В частности, мы
покажем, что достаточно изучить случай $p_1\equiv p_4\equiv0$, сформулируем определение регулярных и сильно регулярных по Биркгофу краевых условий
для случая системы Дирака и приведем несколько элементарных фактов об операторе $\mathcal L_{0,U}$ с нулевым потенциалом. Во втором параграфе мы
получаем асимптотические формулы для собственных значений и собственных функций оператора $\mathcal L_{P,U}$. Отметим, что здесь мы рассматриваем
только общий случай $P\in L_1$, хотя наш метод позволяет уточнить оценки остаточных членов в этих формулах для случая $P\in L_p$ (случай $p\in[1,2]$
разобран в работе \cite{SavSh14}) и для шкалы пространств Бесова $P\in B_{1,q}^\theta$, $q\in[1,\infty]$, $\theta\geqslant0$ (этот случай авторы
планируют рассмотреть в отдельной работе). Третий параграф посвящен изучению функции Грина оператора $\mathcal L_{P,U}$. Мы найдем ее явный вид в
терминах фундаментальной системы решений и докажем ограниченность этой функции в полуплоскостях $|\mathrm{Im\,}\lambda|>\alpha$. Следуя работе
\cite{Ke}, мы построим здесь систему собственных и присоединенных функций оператора $\mathcal L_{P,U}$. Также здесь доказана теорема об
асимптотическом поведении спектральных проекторов. Результаты о полноте и минимальности системы собственных и присоединенных векторов оператора
$\mathcal L_{P,U}$ приведены в четвертом параграфе работы. Отметим, что эти результаты уже были кратко изложены в работе \cite{SavSh14}; здесь мы
снабдим их полным доказательством. В четвертом параграфе приведены также результаты о базисности Рисса для случая сильно регулярных краевых условий.
По сравнению с работой \cite{SavSh14} здесь мы несколько модифицировали и упростили доказательство. Наконец, в пятом параграфе работы получен
основной результат работы: доказана базисность Рисса из двумерных подпространств для случая произвольного суммируемого потенциала и регулярных, но не
сильно регулярных краевых условий. Этот факт был анонсирован в \cite{SavSh14}; здесь мы приводим его полное доказательство.

\section{Обозначения и предварительные результаты.}

Заметим, что существует два альтернативных вида записи системы Дирака. В данной работе мы будем рассматривать систему вида
\begin{gather}\label{eq:0.lP}
\ell_P(\mathbf y)=B\mathbf y'+P\mathbf y,\quad \text{где}\\
 B = \begin{pmatrix} -i & 0 \\ 0 & i
        \end{pmatrix},
        \qquad
    P(x) = \begin{pmatrix} p_1(x) & p_2(x) \\ p_3(x) & p_4(x)
        \end{pmatrix},
        \qquad
     \mathbf y(x)=\begin{pmatrix}y_1(x)\\ y_2(x)\end{pmatrix}\notag
\end{gather}
в пространстве $\mathbb H=L_2[0,\pi]\oplus L_2[0,\pi]\ni \mathbf y$. Функции $p_j$, $j=1,2,3,4$, предполагаются суммируемыми на отрезке $[0,\pi]$ и
комплекснозначными. Краевые условия и область определения оператора будут обсуждаться ниже. Другой формой записи (см., например, \cite{LS}) является
\begin{gather}\label{eq:0.lQ}
\ell_Q(\mathbf u)=B\mathbf u'+Q\mathbf u,\quad\text{где }\\ B = \begin{pmatrix} 0 & -1 \\ 1 & 0
        \end{pmatrix},
        \quad
    Q(x) = \begin{pmatrix} q_1(x) & q_2(x) \\ q_3(x) & q_4(x)
        \end{pmatrix},
        \quad
     \mathbf u(x)=\begin{pmatrix}u_1(x)\\ u_2(x)\end{pmatrix}.\notag
\end{gather}
Эти формы записи эквивалентны. Так, замена $u_1=\frac12(y_1+y_2)$, $u_2=\frac{i}2(y_1-y_2)$ сводит систему \eqref{eq:0.lQ} к виду \eqref{eq:0.lP}. Далее мы покажем,
что достаточно изучить случай, когда $p_4=p_1=0$ (для системы, записанной в форме \eqref{eq:0.lQ} это эквивалентно равенствам $q_1=-q_4$, $q_2=q_3$).

Через $\mathbf y(x)=(y_1(x),y_2(x))^t$ будем обозначать вектор--функции на отрезке $[0,\pi]$, а через
$$
\langle\mathbf f,\,\mathbf g\rangle=\int_0^\pi(f_1(x)\overline{g_1}(x)+f_2(x)\overline{g_2}(x))\,dx
$$
--- скалярное произведение в пространстве $\mathbb H$. Чтобы не
усложнять запись, мы будем   писать $\mathbf f \in L_p$, имея в виду, что  $\mathbf f \in L_p[0,\pi]\times L_p[0,\pi]$, или $P\in L_p$, имея в виду, что все
компоненты матрицы лежат в $L_p$. Норму по переменной $x\in[0,\pi]$ в пространстве $L_p$ или в $L_p\times L_p$ будем обозначать $\|\cdot\|_p$.

Перейдем к определению оператора $\mathcal L_P$, который мы свяжем с дифференциальным выражением $\ell_P$. Прежде всего, определим максимальный оператор
$$
\mathcal L_{P,M}\, \mathbf y:=\ell_P(\mathbf y);\qquad\mathfrak D(\mathcal L_{P,M})=\{\mathbf y\in AC[0,\pi]: \ell_P(\mathbf y)\in\mathbb H\}
$$
и минимальный оператор $\mathcal L_{P,m}$, являющийся сужением оператора $\mathcal L_{P,M}$ на область
$$
\mathfrak D(\mathcal L_{P,m})=\{\mathbf y\in\mathfrak D(\mathcal L_{P,M}): \mathbf y(0)=\mathbf y(\pi)=0\}.
$$
Здесь $AC[0,\pi]=W_1^1[0,\pi]$ --- пространство абсолютно непрерывных функций.   Поскольку элементы матрицы $P$ --- суммируемые функции, оба слагаемых
дифференциального выражения $\ell_P(\mathbf y)$  корректно определены, как функции из $L_1$. При этом, в область определения оператора входят только те функции
$\mathbf y$, для которых сумма этих слагаемых принадлежит $\mathbb H$.  Через $\mathcal L_{P^* , M}$ и $\mathcal L_{P^* , m}$ будем обозначать максимальный и
минимальный операторы, порожденные сопряженным дифференциальным выражением
$$
\ell_{P^*}(\mathbf y):=B\mathbf y'+P^*\mathbf y,\qquad\text{где}\
P^*=\begin{pmatrix}\overline{p_1} & \overline{p_3}\\
\overline{p_2} & \overline{p_4}\end{pmatrix}.
$$

\begin{proposition}[Формула Лагранжа]\label{lem:0.1}
Для любых функций \(\mathbf f\in\mathfrak D(\mathcal L_{P,M})\), \(\mathbf g\in\mathfrak D(\mathcal L_{P^*, M})\) справедливо тождество
\begin{equation}\label{eq:lagr}
    \langle\mathcal L_{P,M}\mathbf f,\mathbf g\rangle=\langle\mathbf f,\mathcal L_{P^*,M}\mathbf g\rangle+[\mathbf f,\mathbf g]^{\pi}_0,\qquad\text{где}\
    [\mathbf f,\mathbf g]_0^{\pi}=-i\left.f_1(x)\overline{g_1}(x)\right|_0^{\pi}+
i\left.f_2(x)\overline{g_2}(x)\right|_0^{\pi}.
\end{equation}
\end{proposition}
\begin{proof}
Равенство \eqref{eq:lagr} получается интегрированием по частям.
\end{proof}
Из этой формулы, в частности, получаем
\begin{equation}\label{eq:0.lagrange}
    \langle\mathcal L_{P,M}\mathbf f,\mathbf g\rangle=\langle\mathbf f,\mathcal L_{P^*,m}\mathbf g\rangle,\qquad\mathbf f\in
    \mathfrak D(\mathcal L_{P,M}),\;\mathbf g\in\mathfrak D(\mathcal L_{P^*,
m}).
\end{equation}

\vskip 0,2cm

В дальнейшем важную роль играет следующее утверждение, которое легко следует из известного результата теории обыкновенных дифференциальных уравнений
(см., например,  \cite [Гл. III \S2] {CL}).

\begin{THEOREM}\label{A}
Пусть $\mathbf A(x)$ --- матрица размера $n\times n$, элементы которой являются функциями пространства $L_1[0,\pi]$, а $\mathbf f\in \big [ L_1[0,\pi]\big ]^n$ ---
вектор-функция. Тогда при любом $c\in[0,\pi]$ уравнение
$$
\mathbf y'=\mathbf A(x)\mathbf y+\mathbf f,\text{ с условием } \mathbf y(c)=\mathbf{\xi}\in\mathbb C^n,
$$
имеет единственное решение $\mathbf y(\cdot)\in AC[0,\pi]$.
\end{THEOREM}

Напомним, что оператор \(F\), действующий в гильбертовом (или банаховом) пространстве \(H\), называется \textit{фредгольмовым}, если его область
определения плотна в \(H\), образ замкнут, а дефектные числа \(\{\alpha,\beta\}\), равные размерностям ядра и коядра, конечны.

Из утверждения \ref{lem:0.1} и теоремы \ref{A} сразу следует

\begin{proposition}\label{tm:0.2} При любом \(\lambda\in\mathbb C\) операторы
\(\mathcal L_{P, M}-\lambda I\) и \(\mathcal L_{P^*, m} -\overline{\lambda} I\) фредгольмовы, являются взаимно сопряженными, а их дефектные числа равны \(\{2,0\}\) и
\(\{0,2\}\), соответственно.
\end{proposition}

Перейдем к описанию расширений $\mathcal L$ оператора $\mathcal L_{P,m}$, для которых $\mathcal L_{P,m}\subset\mathcal L\subset\mathcal L_{P,M}$. Заметим, что любой
такой оператор имеет область определения
$$
\mathfrak D(\mathcal L)=\{\mathbf y\in\mathfrak D(\mathcal L_{P,M}): U_j(\mathbf y)=0, 1\leqslant j\leqslant\nu\},
$$
где $U_j$ --- линейные формы от векторов $\mathbf y(0)$ и $\mathbf y(\pi)$. Эти формы можно считать линейно независимыми и тогда их число $\nu$ заключено между $0$ и
$2$. Если мы хотим, чтобы оператор $\mathcal L$ имел непустое резольвентное множество, т.е. для некоторого $\lambda\in\mathbb C$ индексы оператора $\mathcal L-\lambda
I$ были нулевыми, то, согласно утверждению \ref{tm:0.2}, $\nu=1$. Таким образом, оператор $\mathcal L=\mathcal L_{P,U}$ имеет область определения
\begin{gather}
    \mathcal D(\mathcal L_{P,U})=\left\{\mathbf y\in\mathcal D(\mathcal L_{P,M}): U(\mathbf y)=0\right\},\quad\text{где}\notag\\
    U(\mathbf y)=C\mathbf y(0)+D\mathbf y(\pi)=\begin{pmatrix}u_{11} & u_{12}\\ u_{21} &
u_{22}\end{pmatrix}\begin{pmatrix}y_1(0)\\
y_2(0)\end{pmatrix}+\begin{pmatrix}u_{13} & u_{14}\\ u_{23} &
u_{24}\end{pmatrix}\begin{pmatrix}y_1(\pi)\\
y_2(\pi)\end{pmatrix},\label{matrU}
\end{gather}
причем строки матрицы
$$
\mathcal U:=(C,\,D)=\begin{pmatrix}u_{11}&u_{12}&u_{13}&u_{14}\\
    u_{21}&u_{22}&u_{23}&u_{24}\end{pmatrix}
$$
линейно независимы. Обозначим через \(J_{\alpha\beta}\) определитель, составленный из \(\alpha\)-го и \(\beta\)-го столбца матрицы $\mathcal U$.
\begin{definition}\label{def:reg}
Краевое условие, определенное формой $U$, называется {\it регулярным} (по Биркгофу), если $J_{14}\cdot J_{23}\ne0$. Оператор Дирака, порожденный регулярным краевым
условием $U$ (т.е. оператор $\mathcal L_{P,U}$  с областью определения \eqref{matrU}), будем называть {\it регулярным}.
\end{definition}
Далее в работе мы будем рассматривать только регулярные краевые условия, так как для данной задачи регулярные операторы сохраняют классические асимптотики для
собственных значений и собственных функций.

Мы уже говорили выше, что без ограничения общности можно считать функции $p_1$ и $p_4$ нулевыми. Сформулируем соответствующее утверждение. Вначале напомним, что если
два оператора $A_1$ и $A_2$ в гильбертовом пространстве с плотными областями определения подобны, т.е. существует такой ограниченный и ограниченно обратимый оператор
$T$, что $A_2=T^{-1}A_1T$, а $\mathfrak D(A_2)=T^{-1}\mathfrak D(A_1)$, то из замкнутости одного оператора следует замкнутость другого. Подобные операторы имеют
одинаковый спектр, в частности, если спектр оператора $A_1$ состоит из собственных значений $\sigma(A_1)=\{\lambda_n\}$, то и $\sigma(A_2)=\{\lambda_n\}$, причем
кратности этих собственных значений для $A_1$ и $A_2$ совпадают. Если $\{e_n\}$
--- система собственных и присоединенных векторов оператора $A_1$,
то $\{T^{-1}e_n\}$ --- система собственных и присоединенных векторов оператора $A_2$. Отсюда следует, что эти системы обладают одинаковыми
геометрическими свойствами (полнота, минимальность, базисность Рисса, базисность Рисса со скобками и т.д.).
\begin{proposition}\label{lem:0.2} Пусть $P(x)$ --- произвольная матрица размера $2\times2$
с элементами $p_j\in L_1[0,\pi]$, $j=1,\,2,\,3,\,4$, а матрица $\mathcal{U}$ задает регулярные краевые условия. Тогда оператор $\mathcal L_{P,U}$ подобен оператору
$\mathcal L_{\widetilde P,\widetilde U}+\gamma I$, где
\begin{gather}
\widetilde P(x)=\begin{pmatrix}0&\widetilde p_2(x)\\\widetilde p_3(x)&0
\end{pmatrix},\notag\\
\widetilde p_2(x)=p_2(x)e^{i(\varphi(x)-\psi(x))},\qquad \widetilde
p_3(x)=p_3(x)e^{i(\psi(x)-\varphi(x))},\label{sim}\\
\varphi(x)=\gamma x-\int_0^x p_1(t)dt,\qquad \psi(x)=\int_0^x p_4(t)dt-\gamma x,\qquad
\gamma=\frac1{2\pi}\int_0^\pi(p_1(t)+p_4(t))dt,\notag\\
\widetilde{\mathcal{U}}=(\widetilde C,\,\widetilde D),\quad \widetilde C=C,\quad\widetilde D=\exp\left(\frac{i}2\int_0^\pi(p_1(t)-p_4(t))dt\right)D.\notag
\end{gather}
\end{proposition}
\begin{proof}
Рассмотрим в пространстве $\mathbb H$ оператор умножения на матрицу $W(x)$
$$
W:\mathbf y\mapsto\begin{pmatrix}e^{i\varphi(x)} & 0\\
0 &e^{i\psi(x)}\end{pmatrix}\begin{pmatrix}y_1(x)\\
y_2(x)\end{pmatrix}.
$$
Заметим, что этот оператор ограничен, поскольку функции $\varphi$
и $\psi$ абсолютно непрерывны, и ограниченно обратим. Тогда
\begin{multline*}
W^{-1}\ell_{P}(W\mathbf f)=W^{-1}BW\mathbf f'+\left(W^{-1}PW+W^{-1}BW'\right)\mathbf f=\\
=\begin{pmatrix}-i&0\\0&i\end{pmatrix}\mathbf f'+
\begin{pmatrix}p_1&p_2e^{i(\psi-\varphi)}\\p_3e^{i(\varphi-\psi)}&p_4\end{pmatrix}\mathbf f+
\begin{pmatrix}\varphi'&0\\0&-\psi'\end{pmatrix}\mathbf f=\ell_{\widetilde P}(\mathbf f)+\gamma\mathbf f.
\end{multline*}
Остается найти область определения оператора $W^{-1}\mathcal L_{P,U}W$. Заметим, что если $\mathbf y\in AC[0,\pi]$, то и $W^{-1}\mathbf y\in AC[0,\pi]$; если
$\ell_P(\mathbf y)\in\mathbb H$, то и $\ell_{\widetilde P}(W^{-1}\mathbf y)=W^{-1}\ell_P(\mathbf y)\in\mathbb H$, так что для максимального оператора $W^{-1}\mathfrak
D(\mathcal L_{P,M})=\mathfrak D(\mathcal L_{\widetilde P,M})$. Легко видеть, что
$$
W(0)=I,\quad \text{а} \quad W(\pi)=\exp\left(\frac{i}2\int_0^\pi(p_4(t)-p_1(t))dt\right)I.
$$
Если $\mathbf z\in\mathfrak D(\mathcal L_{\widetilde P,\widetilde U})$, то $\mathbf z=W^{-1}\mathbf y$, где $\mathbf y\in\mathfrak D(\mathcal L_{P,U})$. Тогда краевые
условия принимают вид
\begin{multline*}
C\mathbf y(0)+D\mathbf y(\pi)=0\ \Longleftrightarrow\
CW(0)\mathbf z(0)+DW(\pi)\mathbf z(\pi)=0\ \Longleftrightarrow\\
\Longleftrightarrow C\mathbf z(0)+\exp\left(\frac{i}2\int_0^\pi(p_1(t)-p_4(t))dt\right)D\mathbf z(\pi)=0,
\end{multline*}
т.е.
$$
\widetilde{\mathcal{U}}=(\widetilde C,\,\widetilde D),\quad \text{где}\ \ \widetilde C=C,\ \text{а}\ \widetilde
D=\exp\left(\frac{i}2\int_0^\pi(p_1(t)-p_4(t))dt\right)D.
$$
\end{proof}
Всюду далее в работе мы будем считать, что преобразования уже проведены (при этом спектральный параметр $\lambda$ мы заменяем на $\lambda+\gamma$). Таким образом, мы
будем рассматривать оператор, порожденный дифференциальным выражением \eqref{eq:0.lP}, где матрица $P(x)$ имеет вид
\begin{equation} \label{P}
P(x)=\begin{pmatrix}0&p_2(x)\\p_3(x)&0\end{pmatrix}, \qquad p_2(x),\ p_3(x)\in L_1[0, \pi],
\end{equation}
и регулярными краевыми условиями \eqref{matrU}.
\begin{definition}\label{def:2}
Оператор Дирака $\mathcal L_{P,U}$ называется {\it сильно регулярным}, если он регулярен и к тому же $(J_{12}+J_{34})^2+4J_{14}J_{23}\ne0$.
\end{definition}

Мы будем сравнивать асимптотическое поведение собственных значений и собственных функций оператора $\mathcal L_{P,U}$ и оператора $\mathcal L_{0,U}$. Рассмотрим
оператор $\mathcal L_{0,U}$, порожденный дифференциальным выражением $\ell_0(\mathbf y)=B\mathbf y'$ и регулярным краевым условием $U(\mathbf y)=0$ вида
\eqref{matrU}.
\begin{proposition}\label{tm:0.3}
Спектр оператора $\mathcal L_{0,U}$ состоит из собственных значений, которые можно записать двумя сериями $-\frac{i}{\pi}\ln z_0+2n$ и $-\frac{i}{\pi}\ln z_1+2n$,
$n\in\mathbb Z$, где $z_0$ и $z_1$ --- корни квадратного уравнения
\begin{equation}\label{sqeq}
J_{23}z^2-[J_{12}+J_{34}]z-J_{14}=0,
\end{equation}
а значения ветви логарифма фиксируются в полосе $\mathrm{Im\,}\, z\in(-\pi,\pi]$.
\end{proposition}
В дальнейшем мы будем нумеровать эти собственные значения одним индексом $n\in\mathbb Z$, объединяя две серии в одну:
\begin{equation}\label{eq:la0}
\lambda^0_n=\begin{cases} \varkappa_0+n,\quad \text{для четных }n,\\ \varkappa_1+n,\quad \text{для нечетных }n,\end{cases}\quad\text{где}\quad
\varkappa_0=-\frac{i}{\pi}\ln z_0,\quad \varkappa_1=-\frac{i}{\pi}\ln z_1-1,
\end{equation}
причем $-1<\mathrm{Re\,}\varkappa_0\leqslant\mathrm{Re\,}\varkappa_1+1\leqslant1$. В случае $\mathrm{Re\,}\varkappa_0=\mathrm{Re\,}\varkappa_1+1$ для определенности
будем считать, что $\mathrm{Im\,}\varkappa_0\leqslant\mathrm{Im\,}\varkappa_1$.

Доказательство этого утверждения, так же как и другие сведения об операторе $\mathcal L_{0,U}$, можно найти в работе П.~Джакова и Б.~Митягина \cite{DM2}. Мы, однако,
приведем их здесь для удобства читателя.
\begin{proof}
Решениями уравнения $\ell_0(\mathbf y)=\lambda\mathbf y$ с начальными условиями $(1,\,0)^t$ и $(0,\,1)^t$ являются функции $\mathbf e_1^0(x,\lambda)=(e^{i\lambda
x},\,0)^t$ и $\mathbf e_2^0(x,\lambda)=(0,\, e^{-i\lambda x})^t$ соответственно, а общее решение имеет вид $\mathbf y=\omega_1^0\mathbf e_1^0+\omega_2^0\mathbf
e_2^0$. Подставляя это выражение в краевые условия получаем систему
\begin{equation}\label{eigeneq}
\begin{cases}[u_{11}+u_{13}e^{i\pi\lambda}]\omega_1^0+[u_{12}+u_{14}e^{-i\pi\lambda}]\omega_2^0=0,\\
[u_{21}+u_{23}e^{i\pi\lambda}]\omega_1^0+[u_{22}+u_{24}e^{-i\pi\lambda}]\omega_2^0=0.\end{cases}
\end{equation}
Обозначим матрицу этой системы через $M_0(\lambda)$. Число $\lambda\in\mathbb C$ является собственным значением оператора $\mathcal L_{0,U}$ тогда и только тогда,
когда определитель $\Delta_0(\lambda): = \det M_0(\lambda)$ обращается в ноль. Непосредственными вычислениями получаем
\begin{equation}\label{Delta0}
\Delta_0(\lambda)=[J_{12}+J_{34}]-J_{23}e^{i\pi\lambda}+J_{14}e^{-i\pi\lambda}.
\end{equation}
Остается сделать в этом уравнении подстановку $e^{i\pi\lambda}=z$.
\end{proof}
\begin{proposition}\label{tm:0.4}
Нормированные собственные функции $\mathbf y^0_n$, $n\in\mathbb Z$, сильно регулярного оператора $\mathcal L_{0,U}$ имеют вид
\begin{equation}\label{eigenfunc}
\mathbf y^0_n=\omega^0_{1,j}\big(e^{i\lambda^0_nx},\,0\big)^t+ \omega^0_{2,j}\big(0,\,e^{-i\lambda^0_nx}\big)^t,\quad n\in\mathbb Z,\quad\text{где}
\end{equation}
$j=0$ при четном $n$ и $j=1$ при нечетном $n$. Числа $\omega^0_{i,j}$, где $i=1,\,2$, а $j=0,\,1$, определяются матрицей $\mathcal U$.
\end{proposition}
\begin{proof}
Собственные функции, введенные в доказательстве предыдущего утверждения, имеют вид $\omega_{1,n}^0\mathbf e^0_1(x,\lambda^0_n)+\omega_{2,n}^0\mathbf
e^0_2(x,\lambda^0_n)$. При этом числа $\omega_{1,n}^0$ и $\omega_{2,n}^0$ есть решения системы \eqref{eigeneq}, в которой $\lambda=\lambda^0_n$. Поскольку матрица
этой системы $2$--периодична по параметру $\lambda$, а $\lambda_{n+2}^0-\lambda_n^0=2$, то числа $\omega_{1,n}^0$ и $\omega_{2,n}^0$ зависят лишь от четности индекса
$n$. Обозначим их $\omega_{1,j}^0$ и $\omega_{2,j}^0$, где $j=0$ при четном $n$ и $j=1$ при нечетном $n$. Остается нормировать собственные функции. Так как $\mathbf
e^0_1(x,\lambda^0_n)=(e^{i\lambda_n^0x},0)^t$ и $\mathbf e^0_2(x,\lambda^0_n)=(0,e^{-i\lambda_n^0x})^t$, то
\begin{multline*}
\left\|\omega^0_{1,j}\mathbf e^0_1(x,\lambda^0_n)+ \omega^0_{2,j}\mathbf e^0_2(x,\lambda_n^0)\right\|_{\mathbb H}^2=
|\omega^0_{1,j}|^2\int_0^\pi|e^{i\lambda_n^0x}|^2\,dx+|\omega^0_{2,j}|^2\int_0^\pi|e^{-i\lambda_n^0x}|^2\,dx=\\
=|\omega^0_{1,j}|^2\int_0^\pi|e^{i\varkappa_j^0x}|^2\,dx+|\omega^0_{2,j}|^2\int_0^\pi|e^{-i\varkappa_j^0x}|^2\,dx.
\end{multline*}
Последнее выражение зависит только от четности $n$, а значит, после нормировки получим \eqref{eigenfunc} с некоторыми новыми $\omega^0_{i,j}$, $i=1,\,2$, которые
по--прежнему зависят только от четности $n$.
\end{proof}
\begin{remark}
Если оператор $\mathcal L_{0,U}$ сильно регулярен, то дискриминант квадратного уравнения \eqref{sqeq} отличен от нуля и корни $z_0,\, z_1$ различны. Корневые
подпространства регулярного, но не сильно регулярного оператора $\mathcal L_{0,U}$, отвечающие каждому собственному значению, двумерны. При этом возможны два случая
--- либо в каждом подпространстве есть базис из двух собственных
функций оператора $\mathcal L_{0,U}$, либо каждое подпространство содержит ровно один (с точностью до множителя) собственный вектор.
\end{remark}

\section{Асимптотические формулы.}

Обозначим через
\begin{equation}\label{matrE}
E(x,\lambda)=\begin{pmatrix}e_{11}(x,\lambda)&e_{12}(x,\lambda)\\ e_{21}(x,\lambda)&e_{22}(x,\lambda)\end{pmatrix},\qquad \mathbf
e_1(x,\lambda)=\begin{pmatrix}e_{11}(x,\lambda)\\ e_{21}(x,\lambda)\end{pmatrix},\ \mathbf e_2(x,\lambda)=\begin{pmatrix}e_{12}(x,\lambda)\\
e_{22}(x,\lambda)\end{pmatrix},
\end{equation}
матрицу фундаментальной системы решений уравнения $\ell_P(\mathbf y)=\lambda\mathbf y$ с начальными условиями $E(0,\lambda)=I$. Для исследования регулярного оператора
$\mathcal L_{P,U}$ мы воспользуемся результатами об асимптотическом поведении фундаментальной системы решений \eqref{matrE} в комплексной $\lambda$--плоскости внутри
полос $\Pi_\alpha=\{\lambda\in\mathbb C\,\vert\ |\mathrm{Im\,} \lambda|<\alpha\}$, полученными авторами в \cite{SaSa1}. В этой работе рассматривался оператор
$\mathcal L_{Q,U}$, записанный в форме \eqref{eq:0.lQ}, а оценки остаточных членов в асимптотических формулах были получены для потенциала $Q$ из пространств $L_p$,
$p\in[1,\infty]$. Здесь нам потребуются только результаты для случая $p=1$, причем мы переформулируем их для системы Дирака, записанной в форме \eqref{eq:0.lP} с
матрицей \eqref{P}.

Положим
\begin{align}
&e_{11}(x,\lambda)=e^{i\lambda x}+\rho_{11}(x,\lambda),\qquad & &e_{21}(x,\lambda)=\rho_{21}(x,\lambda),\notag\\
&e_{12}(x,\lambda)=\rho_{12}(x,\lambda),\qquad & &e_{22}(x,\lambda)=e^{-i\lambda x}+\rho_{22}(x,\lambda).\label{eq:rhodef}
\end{align}
Заметим сразу, что $\rho_{j,k}(0,\lambda)=0$, $j,\,k\in\{1,2\}$.

\begin{THEOREM}\label{B}
Пусть $P(x)$ имеет вид \eqref{P}, а $\alpha>0$
--- произвольное фиксированное число. Тогда
\begin{align}
\rho_{j,k}(x,\lambda)\to0\ j,\,k\in\{1,\,2\},\quad\text{при }\Pi_\alpha\ni\lambda\to\infty\label{stripas}
\end{align}
равномерно по $x\in[0,\pi]$. Более того, найдется такое число $\beta=\beta(P,\alpha)>0$, что для всех $\lambda\in\Pi_{\alpha,\beta}:=\left\{\lambda\in\Pi_\alpha:
|\mathrm{Re\,}\lambda|>\beta\right\}$
\begin{equation}\label{stripas2}
\rho_{1k}(x,\lambda)=\eta_{1k}(x,\lambda)e^{i\lambda x},\qquad\rho_{2k}(x,\lambda)=\eta_{2k}(x,\lambda)e^{-i\lambda x},\quad k=1,\,2,
\end{equation}
причем почти всюду на $[0,\pi]$ выполнены оценки
\begin{align}
&\sup_{\lambda\in \Pi_{\alpha,\beta}}\left|\left(\eta_{11}(x,\lambda)\right)'_x\right|\leqslant M|p_2(x)|,\qquad & &\sup_{\lambda\in
\Pi_{\alpha,\beta}}\left|\left(\eta_{12}(x,\lambda)\right)'_x\right|\leqslant
M|p_2(x)|,\notag\\
&\sup_{\lambda\in \Pi_{\alpha,\beta}}\left|\left(\eta_{21}(x,\lambda)\right)'_x\right|\leqslant M|p_3(x)|,\qquad & &\sup_{\lambda\in
\Pi_{\alpha,\beta}}\left|\left(\eta_{22}(x,\lambda)\right)'_x\right|\leqslant M|p_3(x)|
\end{align}
для некоторого $M=M(P,\alpha)$.
\end{THEOREM}

Асимптотическое поведение функций $\mathbf e_j(x,\lambda)$ вне полос $\Pi_\alpha$ в работе \cite{SaSa1} не изучалось. Применяя метод, аналогичный методу,
использованному в этой работе, несложно получить асимптотические представления для $\mathbf e_j(x,\lambda)$, $j=1,\,2$, в секторах
$$
S_1=\{\lambda\in\mathbb C: \varepsilon<\arg\lambda<\pi-\varepsilon\}\quad\text{и}\quad S_2=\{\lambda\in\mathbb C: -\pi+\varepsilon<\arg\lambda<-\varepsilon\},
$$
где $\varepsilon\in(0,\pi/2)$ произвольно. Более того, можно получить квалифицированную оценку остаточных членов в этих представлениях в зависимости от  индекса $p$
пространства $L_p[0,\pi]\ni p_j$, $j=2,\,3$. Здесь, однако, нас интересует только случай $p=1$, и потому мы воспользуемся результатом работы \cite{MalOri}. В ней
изучался случай общей системы $B\mathbf y'+P\mathbf y$ в пространстве $(L_2[0,\pi])^n$. Мы сформулируем здесь теорему 2.2 этой работы для случая системы Дирака.

\begin{THEOREM}\label{C}
Пусть матрица $P(x)$ имеет вид \eqref{P}. Существует матрица $Y(x,\lambda)$ фундаментальной системы решений уравнения $\ell_P(\mathbf y)=\lambda\mathbf y$, элементы
которой $y_{jk}(x,\lambda)$ являются целыми функциями параметра $\lambda$  с ограничением на рост $|y_{jk}(x,\lambda)|\leqslant Me^{x|\lambda|}$, где $M$ не зависит
от $x$ и $\lambda$. Кроме того, $Y(x,\lambda)$ имеет асимптотическое представление
\begin{equation}\label{eq:Malamud}
Y(x,\lambda)=\begin{pmatrix}e^{i\lambda x}(1+o(1))&e^{-i\lambda x}\cdot o(1)\\
e^{i\lambda x}\cdot o(1)&e^{-i\lambda x}(1+o(1))\end{pmatrix}
\end{equation}
при $\lambda\to\infty$ в секторах $S_1$ и $S_2$ равномерно по $x\in[0,\pi]$.
\end{THEOREM}

Нам необходимо выяснить асимптотическое поведение функций $\mathbf e_1(x,\lambda)$ и $\mathbf e_2(x,\lambda)$ при $\lambda\to\infty$ во всей комплексной плоскости.
Для этого мы воспользуемся теоремами \ref{B} и \ref{C}, а также фактом из теории целых функций, сформулированным ниже. Этот факт хорошо известен специалистам, но
для полноты изложения мы приведем его с доказательством, опираясь на следующее утверждение (см. \cite{Lin}).
\begin{THEOREM}\label{tm:FrLin}
Пусть $D\in\mathbb C$ --- ограниченная область, а функция $f(z)$ голоморфна в $D$ и непрерывна в $\overline{D}$. Пусть далее $\zeta\in D$, $U_\rho=\{z\in\mathbb C:
|z-\zeta|\leqslant\rho\}$, причем на окружности $|z-\zeta|=\rho$ имеется дуга, не принадлежащая $D$, длина которой $l\geqslant2\pi\rho/n$ для некоторого натурального
$n$. Пусть $|f(z)|\leqslant M_0$ для всех $z\in\partial D\cap \overline{U_\rho}$ и $|f(z)|\leqslant M$ для всех остальных точек $z\in\partial D$. Тогда
$|f(\zeta)|\leqslant M_0^{1/n}M^{1-1/n}$.
\end{THEOREM}
\begin{lemma}\label{FrLin}
Пусть $f$ --- целая функция, а $S=\{z\in\mathbb C: \arg z\in(0,\,\pi/2)\}$. Обозначим
\begin{gather*}
M_0(r)=\sup\{|f(z)|: \arg z=0,\ |z|\geqslant r\},\\
M_1(r)=\sup\{|f(z)|: \arg z\in[0,\pi/6],\ |z|\geqslant r\},\\
M=\sup\{|f(z)|: z\in\overline{S}\}.
\end{gather*}
Тогда $M_1(r)\leqslant M_0(r/4)^{1/4}\cdot M^{3/4}$.
\end{lemma}
\begin{proof}
Возьмем точку $\zeta=re^{i\alpha}$, где $\alpha\in[0,\pi/6]$. Случай $\alpha=0$ тривиален, поскольку $M_0(r)\leqslant M_1(r)$, так что далее считаем $\alpha>0$.
Рассмотрим окружность радиуса $\rho=r\sqrt{2}\sin\alpha$ с центром в точке $\zeta$. Легко видеть, что эта окружность лежит в правой полуплоскости, пересекает
вещественную ось в точках $x_\pm=r(\cos\alpha\pm\sin\alpha)$, причем отрезок $[x_-,x_+]$ виден из точки $\zeta$ под прямым углом. Применим теорему \ref{tm:FrLin} к
функции $f(z)$ и области $D=\{z: |z-\zeta|<\rho\}\cap S$. Тогда
$$
|f(\zeta)|\leqslant\left(\max_{z\in[x_-,x_+]}|f(z)|\right)^{1/4}\cdot M^{3/4}.
$$
Учитывая, что $\alpha\leqslant\pi/6$, получаем, что $x_->r/4$, откуда следует
утверждение леммы.
\end{proof}
\begin{theorem}\label{tm:0.6}
Функции $e_{ij}(x,\lambda)$ аналитичны по $\lambda$ во всей комплексной плоскости и
\begin{equation}\label{mainasc}
E(x,\lambda)=\left(\begin{array}{cc}e^{i\lambda x}\cdot(1+o(1))+e^{-i\lambda x}\cdot o(1)&\quad e^{i\lambda x}\cdot o(1)+e^{-i\lambda x}\cdot o(1)\\
e^{i\lambda x}\cdot o(1)+e^{-i\lambda x}\cdot o(1)&\quad e^{i\lambda x}\cdot o(1)+e^{-i\lambda x}\cdot (1+o(1))
\end{array}\right)
\end{equation} при $\mathbb C\ni\lambda\to\infty$ равномерно по $x\in[0,\pi]$.
\end{theorem}
\begin{proof}
Из теоремы \ref{C} следует, что функции $e_{jk}(x,\lambda)$ являются целыми функциями с ограничением на рост $|e_{jk}(x,\lambda)|\leqslant Me^{x|\lambda|}$. Матрица
$E(x, \lambda)$, определенная в \eqref{matrE}, имеет вид
$$
E(x, \lambda)=Y^{-1}(0,\lambda)Y(x,\lambda).
$$
Тогда из \eqref{eq:Malamud} следует \eqref{mainasc} при $S_j\ni\lambda\to\infty$ равномерно по $x\in[0,\pi]$. В то же время, представление \eqref{stripas} влечет
\eqref{mainasc} на лучах $\arg\lambda=0$ и $\arg\lambda=\pi$. Для завершения доказательства теоремы нам достаточно показать, что представление \eqref{mainasc}
справедливо также в секторах
\begin{gather*}
S_3=\{\lambda\in \mathbb C : \arg\lambda\in[0, \pi/6]\},\quad S_4=\{\arg\lambda\in[5\pi/6, \pi]\},\\
S_5=\{\arg\lambda\in[-\pi, -5\pi/6]\},\quad\text{и}\quad S_6=\{\arg\lambda\in[-\pi/6, 0]\}.
\end{gather*}
Рассмотрим сектор $S_3$ (остальные три случая разбираются аналогично). Пусть $\rho_{jk}(x,\lambda)$, $j,\,k=1,\,2$, --- функции, введенные в \eqref{stripas}.
Зафиксируем произвольную пару индексов $j$, $k$ и точку $x\in[0,\pi]$ и обозначим $f(z)=\rho_{jk}(x,z)e^{ixz}$. Тогда $f(z)$ является целой функцией, причем
$|f(z)|\leqslant M$ в секторе $S_3$, а на положительном луче вещественной оси $f(z)=o(1)$ при $z\to\infty$ равномерно по $x\in [0, \pi]$. Согласно лемме \ref{FrLin},
$f(z)=o(1)$ в секторе $S_3$ при $|z|\to\infty$ равномерно по $x\in [0, \pi]$.
\end{proof}
Теорема \ref{B} и теорема \ref{tm:0.6} позволяют получить асимптотические формулы для характеристического определителя оператора $\mathcal L_{P,U}$ с потенциалом вида
\eqref{P} и регулярными краевыми условиями.
\begin{definition}\label{def.3}
Пусть потенциал $P\in L_1$, краевые условия заданы матрицей $\mathcal U$, а функции $\mathbf e_1(x,\lambda)$ и $\mathbf e_2(x,\lambda)$ определены в
\eqref{matrE}. {\it Характеристическим определителем} $\Delta(\lambda)$ оператора $\mathcal L_{P,U}$ называется детерминант матрицы
\begin{equation}\label{matrM}
M(\lambda)=\begin{pmatrix}u_{11}+u_{13}e_{11}(\pi,\lambda)+u_{14}e_{21}(\pi,\lambda)&
u_{12}+u_{13}e_{12}(\pi,\lambda)+u_{14}e_{22}(\pi,\lambda)\\
u_{21}+u_{23}e_{11}(\pi,\lambda)+u_{24}e_{21}(\pi,\lambda)& u_{22}+u_{23}e_{12}(\pi,\lambda)+u_{24}e_{22}(\pi,\lambda)
\end{pmatrix}.
\end{equation}
\end{definition}
\begin{proposition}\label{tm:0.7}
Пусть потенциал $P$ имеет вид \eqref{P}, а краевые условия $U$ регулярны. Пусть $\Delta(\lambda)$ --- характеристический определитель оператора $\mathcal L_{P,U}$, а
$\Delta_0(\lambda)$ --- характеристический определитель оператора $\mathcal L_{0,U}$. Тогда при $\lambda\to\infty$ в произвольной полосе $\Pi_\alpha$ справедливо
асимптотическое представление
$$
\Delta(\lambda)=\Delta_0(\lambda)+o(1).
$$
Кроме того, найдется такая полоса $\Pi_{\alpha_0}$, что при $\lambda\to\infty$ вне этой полосы, справедливо асимптотическое представление
$$
\Delta(\lambda)=\Delta_0(\lambda)(1+o(1)).
$$
\end{proposition}
\begin{proof}
Определитель матрицы $M(\lambda)$ имеет вид
\begin{multline}\label{15}
\Delta(\lambda)=J_{12}+J_{13}e_{12}(\pi,\lambda)+J_{14}e_{22}(\pi,\lambda)+
J_{32}e_{11}(\pi,\lambda)+J_{42}e_{21}(\pi,\lambda)+\\
+J_{34}(e_{11}(\pi,\lambda)e_{22}(\pi,\lambda)-e_{12}(\pi,\lambda)e_{21}(\pi,\lambda))
\end{multline}
(напомним, что через \(J_{\alpha\beta}\) мы обозначаем определитель, составленный из \(\alpha\)-го и \(\beta\)-го столбца матрицы $\mathcal U$). Заметим, что
выражение $e_{11}(x,\lambda)e_{22}(x,\lambda)-e_{12}(x,\lambda)e_{21}(x,\lambda)$ является определителем матрицы фундаментальной системы решений в точке
$x\in[0,\pi]$. Поскольку след матрицы $B^{-1}(\lambda I-P(x))$ равен нулю, то, согласно теореме Лиувилля (см., например, \cite[Гл. III \S1]{CL}), это выражение не
зависит от $x$, а при $x=0$ оно равно единице по определению функций $\mathbf e_1(x,\lambda)$ и $\mathbf e_2(x,\lambda)$.  Подставляя асимптотические формулы
\eqref{mainasc} в соотношение \eqref{15}, получим
$$
\Delta(\lambda)=J_{12}+J_{34}+J_{14}e^{-i\pi\lambda}+J_{32}e^{i\pi\lambda}+ o(1)\left(|e^{i\pi\lambda}|+|e^{-i\pi\lambda}|\right)=
\Delta_0(\lambda)+o(1)\left(|e^{i\pi\lambda}|+|e^{-i\pi\lambda}|\right).
$$
Остаточный член в этом равенстве есть $o(1)$ при $\Pi_\alpha\ni\lambda\to\infty$ для любого $\alpha>0$ и первое утверждение теоремы доказано. Докажем второе
утверждение. Разберем случай $\mathrm{Im\,}\lambda>0$. Подберем $\alpha_0>0$ так, что
$$
|J_{12}|+|J_{34}|+|J_{32}|e^{-\pi\alpha_0}\leqslant|J_{14}|\frac{e^{\pi\alpha_0}}2.
$$
Тогда при $\mathrm{Im\,}\lambda>\alpha_0$
\begin{equation}\label{Delta0est}
|\Delta_0(\lambda)|\geqslant|J_{14}e^{-i\pi\lambda}|-|J_{12}+J_{34}+J_{32}e^{i\pi\lambda}| \geqslant\frac12|J_{14}e^{-i\pi\lambda}|,
\end{equation}
а значит
$$
|\Delta_0(\lambda)|^{-1}\left(|e^{i\pi\lambda}|+|e^{-i\pi\lambda}|\right)\leqslant 2|J_{14}|^{-1}\left(1+\left|e^{-2\pi\alpha_0}\right|\right)\leqslant4|J_{14}|^{-1}.
$$
Итак,
$$
\text{при}\ \mathrm{Im\,}\lambda>\alpha_0:\quad  \Delta(\lambda)=\Delta_0(\lambda)(1+o(1)).
$$
Случай $\mathrm{Im\,}\lambda<0$ разбирается аналогично.
\end{proof}
Теперь мы покажем, что собственные значения оператора $\mathcal L_{P,U}$ асимптотически сближаются с собственными значениями невозмущенного оператора.
\begin{theorem}\label{tm:0.8}
Пусть потенциал $P$ имеет вид \eqref{P} и $\mathcal L_{P,U}$ --- регулярный оператор Дирака. Обозначим  через $\{\lambda_n^0\}$ собственные значения оператора
$\mathcal L_{0,U}$  и через $\lambda_n$ собственные значения оператора $\mathcal L_{P,U}$ с учетом алгебраической кратности. Тогда при подходящей нумерации
последовательности $\{\lambda_n\}_{n\in\mathbb Z}$ (и такая нумерация возможна)
$$
\lambda_n=\lambda_n^0+o(1)\quad \text{при}\ \  |n|\to\infty.
$$
В частности, $\{\lambda_n\}_{n\in\mathbb Z}\subset \Pi_{\alpha_0}$ для некоторого $\alpha_0>0$.
\end{theorem}
\begin{proof}
Обозначим $f(\lambda):=\Delta(\lambda)-\Delta_0(\lambda)$. В силу утверждения \ref{tm:0.7}, найдется $\alpha_0$ такое, что
$$
\frac{|f(\lambda)|}{|\Delta_0(\lambda)|}\to0\quad \text{при }\ \ \lambda\to\infty,\ \lambda\notin\Pi_{\alpha_0}.
$$
Выберем число $\alpha>\alpha_0$ так, чтобы на прямых $|\mathrm{Im\,}\lambda|=\alpha$ было выполнено неравенство $|f(\lambda)|<|\Delta_0(\lambda)|$. Далее, зафиксируем
произвольное число $\mu\in(0,2)$, для которого на прямой $\mathrm{Re\,}\lambda=\mu$ нет нулей функции $\Delta_0(\lambda)$ и обозначим
$$
m=\min\{|\Delta_0(\lambda)|: \mathrm{Re\,}\lambda=\mu\}.
$$
Вновь обращаясь к утверждению \ref{tm:0.7}, видим, что $|f(\lambda)|\to0$ при $\lambda\to\infty$ внутри полосы $\Pi_\alpha$. Тогда найдется такое натуральное $N_1$,
что при всех $\lambda\in\overline{\Pi}_{\alpha}$, $|\mathrm{Re\,}\lambda|\geqslant\mu+2N_1$, выполнено $|f(\lambda)|<m$. Заметим, что функция $\Delta_0(\lambda)$
периодична с периодом $2$, а значит на вертикальных отрезках $\mathrm{Re\,}\lambda=\mu\pm2n$, $n>N_1$, внутри полосы $\Pi_\alpha$ выполнено
$\min|\Delta_0(\lambda)|=m>|f(\lambda)|$. Применим теорему Руше к прямоугольнику, ограниченному прямыми $\mathrm{Im\,}\lambda=\pm\alpha$,
$\mathrm{Re\,}\lambda=\mu\pm2n$, где $n>N_1$ и получим, что функции $\Delta(\lambda)$ и $\Delta_0(\lambda)$ имеют одинаковое (с учетом кратности) число нулей в любом
таком прямоугольнике.

Перейдем к изучению нулей функции $\Delta(\lambda)$ при $\lambda\to\infty$ внутри полосы $\Pi_\alpha$. Зафиксируем число $r$ так, чтобы круги
$U_r(\lambda_n^0)=\{\lambda:|\lambda-\lambda_n^0|\leqslant r\}$, $n\in\mathbb Z$, не пересекались и лежали в полосе $\Pi_\alpha$. Обозначим
$$
m_n=\min\{|\Delta_0(\lambda)|: |\lambda-\lambda_n^0|=r\}.
$$
Поскольку функция $\Delta_0(\lambda)$ периодична, то $m_n\geqslant M$ для некоторого $M>0$. Тогда существует такое натуральное $N_2$, что при
$|\mathrm{Re\,}\lambda|>\mu+2N_2$ на окружностях $|\lambda-\lambda_n^0|=r$ выполнено: $|f(\lambda)|<|\Delta_0(\lambda)|$. По теореме Руше количество нулей функций
$\Delta(\lambda)$ и $\Delta_0(\lambda)$ в каждом круге $U_r(\lambda_n^0)$, $|n|\geqslant 2N_2+2$, совпадает. Теперь мы занумеруем нули функции $\Delta(\lambda)$ в
каждом таком круге так, чтобы их номера совпадали с номерами нулей функции $\Delta_0(\lambda)$ в этом же круге. Из рассуждений, приведенных выше, следует, что
количество нулей функций $\Delta(\lambda)$ и $\Delta_0(\lambda)$, не попавших в объединение этих кругов, конечно и одинаково. Проведем нумерацию оставшихся нулей
функции $\Delta(\lambda)$ в произвольном порядке. Нули функции $\Delta(\lambda)$
--- собственные значения оператора $\mathcal L_{P,U}$ --- мы обозначим
$\{\lambda_n\}_{n\in\mathbb Z}$. Остается заметить, что число $r$ мы можем уменьшать и выбирать сколь угодно малым. Для любого такого $r$ найдется номер $N(r)$, что
при всех $|n|>N(r)$ выполнено $|\lambda_n-\lambda_n^0|<r$. Иными словами, $\lambda_n=\lambda_n^0+o(1)$.
\end{proof}

Теперь докажем теорему об асимптотике собственных функций  сильно регулярного оператора. В случае регулярного, но не сильно регулярного оператора,
собственные значения асимптотически двукратны. В этом случае мы изучим асимптотическое поведение соответствующих двумерных спектральных проекторов
(см. теорему \ref{tm:Sppr} ниже).
\begin{theorem}\label{tm:0.9}
Пусть потенциал $P(x)$ имеет вид \eqref{P}, а  оператор $\mathcal L_{P,U}$ сильно регулярен. Обозначим через $\{\mathbf y_n(x)\}$ нормированные собственные функции
этого оператора, отвечающие собственным значениям $\{\lambda_n\}$,  а через $\{\mathbf y_n^0(x)\}$ ---  нормированные собственные функции оператора $\mathcal
L_{0,U}$, отвечающие собственным значениям $\{\lambda^0_n\}$. Тогда
\begin{equation}\label{efas1}
\mathbf y_n(x)=\mathbf y_n^0(x)+\mathbf{r}_n(x),\qquad \text{где} \ \|\mathbf{r}_n\|_C\to0.
\end{equation}
Более того, справедливо представление
\begin{equation}\label{efas2}
y_{1, n}(x)=e^{i\lambda_nx}\tau_{1, n}(x), \qquad y_{2, n}(x)=e^{-i\lambda_nx}\tau_{2, n}(x),
\end{equation}
причем $|\tau_{j, n}(0)|\leqslant C$, $j=1, 2$, а производные функций $\tau_{j, n}(x)$ подчинены оценке
\begin{equation}
\label{tau} |\tau'_{j, n}(x)|\leqslant C (|p_2(x)|+|p_3(x)|),
\end{equation}
почти всюду на $[0,\pi]\ni x$, где постоянная $C$ не зависит ни от $n$, ни от $x$.
\end{theorem}
\begin{proof}
Поскольку оператор $\mathcal L_{0,U}$ сильно регулярен, то все его собственные значения просты. Обозначим $\delta=\min_{n\neq m}|\lambda_n^0-\lambda_m^0|/2$. Тогда, в
силу теоремы \ref{tm:0.8}, существует номер $N$, такой, что для всех $|n|>N$ в $\delta-$окрестности точки $\lambda_n^0$ лежит ровно одно собственное значение
$\lambda_n$ оператора $\mathcal L_{P,U}$. Из определения собственных значений следует, что $\Delta_0(\lambda_n^0)=0$, где $\Delta_0(\lambda)=\det M_0(\lambda)$,
$$
M_0(\lambda)=\begin{pmatrix}M^0_{11}(\lambda)&M^0_{12}(\lambda)\\M^0_{21}(\lambda)&M^0_{22}(\lambda)\end{pmatrix}=\begin{pmatrix}u_{11}&u_{12}\\
u_{21}&u_{22}\end{pmatrix}+\begin{pmatrix}u_{13}&u_{14}\\ u_{23}&u_{24}\end{pmatrix}
\begin{pmatrix}e^0_{11}(\pi,\lambda)&e^0_{12}(\pi,\lambda)\\ e^0_{21}(\pi,\lambda)&e^0_{22}(\pi,\lambda)\end{pmatrix},
$$
Обозначим $\omega_n^0=(M^0_{12}(\lambda_n^0),\,-M^0_{11}(\lambda_n^0))^t$  --- тогда функция
$$
\widetilde{\mathbf y}_n^0(x)=\omega_{1, n}^0\mathbf e_1^0(x,\lambda_n^0)+\omega_{2, n}^0\mathbf e_2^0(x,\lambda_n^0)
$$
является собственной (ненормированной) функцией для оператора $\mathcal L_{0,U}$. Для $|n|>N$ аналогично определим вектор
$\omega_n=(M_{12}(\lambda_n),\,-M_{11}(\lambda_n))^t$, так что функция
$$
\widetilde{\mathbf y}_n(x)=\omega_{1, n}\mathbf e_1(x,\lambda_n)+\omega_{2, n}\mathbf e_2(x,\lambda_n)
$$
является собственной для оператора $\mathcal L_{P,U}$. Из \eqref{stripas}
следует, что
$$
\|\mathbf e_1(x,\lambda_n) - \mathbf e_1^0(x,\lambda_n)\|_C+\|\mathbf e_2(x,\lambda_n) - \mathbf e_2^0(x,\lambda_n)\|_C\to0\qquad\text{при }n\to\infty,
$$
а из теоремы \ref{tm:0.8} и явного вида функций $\mathbf e_1^0(x,\lambda)$ и $\mathbf e_2^0(x,\lambda)$
$$
\|\mathbf e_1^0(x,\lambda_n) - \mathbf e_1^0(x,\lambda^0_n)\|_C+\|\mathbf e_2^0(x,\lambda_n) - \mathbf e_2^0(x,\lambda^0_n)\|_C\to0.
$$
Тогда  $\|\omega_n-\omega_n^0\|\longrightarrow0$, а значит
\begin{multline*}
\widetilde{\mathbf y}_n(x)=\omega_{1, n}^0\mathbf e_1(x,\lambda_n)+\omega_{2, n}^0\mathbf e_2(x,\lambda_n)+(\omega_{1, n}-\omega_{1, n}^0)\mathbf
e_1(x,\lambda_n)+(w_{2, n}-\omega_{2,
n}^0)\mathbf e_2(x,\lambda_n)=\\
=\omega_{1, n}^0\mathbf e_1^0(x,\lambda_n)+\omega_{2, n}^0\mathbf e_2^0(x,\lambda_n)+\mathbf{r}_n(x),\qquad\text{где}\quad\|\mathbf{r}_n(x)\|_C\to0.
\end{multline*}
Остается нормировать функции $\widetilde{\mathbf y}_n$ и $\widetilde{\mathbf y}_n^0$. Заметим, что $\big|\|\widetilde{\mathbf y}_n\|_{\mathbb H}-\|\widetilde{\mathbf
y}_n^0\|_{\mathbb H}\big|\leqslant C\|\bold r_n\|_{C}=o(1)$. Далее, функции $\widetilde{\mathbf y}_n^0$ зависят только от четности номера $n$, а значит
последовательность норм $\{\|\widetilde{\mathbf y}_n^0\|\}_{n\in\mathbb N}$ (и в пространстве $C$, и в пространстве $\mathbb H$) отделена от нуля и от бесконечности.
Тогда тем же свойством обладает и последовательность $\{\|\widetilde{\mathbf y}_n\|\}_{\mathbb N}$, откуда
$$
\left\|\frac{\widetilde{\mathbf y}_n(x)}{\|\widetilde{\mathbf y}_n\|_{\mathbb H}}-\frac{\widetilde{\mathbf y}_n^0(x)}{\|\widetilde{\mathbf y}_n^0\|_{\mathbb
H}}\right\|_{C}\leqslant \frac{\|\widetilde{\mathbf y}_n^0\|_{\mathbb H}\cdot\|\widetilde{\mathbf y}_n-\widetilde{\mathbf y}_n^0\|_C+\|\widetilde{\mathbf
y}_n^0\|_C\cdot\big|\|\widetilde{\mathbf y}_n^0\|_{\mathbb H}-\|\widetilde{\mathbf y}_n\|_{\mathbb H}\big|}{\|\widetilde{\mathbf y}_n\|_{\mathbb
H}\cdot\|\widetilde{\mathbf y}_n^0\|_{\mathbb H}}=o(1)
$$
и представление \eqref{efas1} доказано. Для доказательства представления \eqref{efas2} воспользуемся соотношениями \eqref{eq:rhodef} и \eqref{stripas}. Получим
$$
\begin{cases}
\widetilde{y}_{1,n}(x)=e^{i\lambda_nx}\left(\omega_{1,n}+\omega_{1,n}\eta_{11}(x,\lambda_n)+\omega_{2,n}\eta_{12}(x,\lambda_n)\right),\\
\widetilde{y}_{2,n}(x)=e^{-i\lambda_nx}\left(\omega_{2,n}+\omega_{1,n}\eta_{21}(x,\lambda_n)+\omega_{2,n}\eta_{22}(x,\lambda_n)\right),
\end{cases}
\qquad |n|>N.
$$
Введем обозначения
\begin{equation*}
\widetilde{\tau}_{1,n}(x)=\widetilde{y}_{1,n}(x)e^{-i\lambda_nx}\quad \text{и} \quad\widetilde{\tau}_{2,n}(x)=\widetilde{y}_{2,n}(x)e^{i\lambda_nx}.
\end{equation*}
Тогда $\widetilde{\tau}_{j,n}(0)=\omega_{j,n}=\omega_{j,n}^0+o(1)$ при $|n|\to\infty$, $j=1,\,2$. Поскольку числа $\omega^0_{1,n}$ и $\omega^0_{2,n}$ зависят только
от четности номера $n$, то последовательности $\{\widetilde{\tau}_{1,n}(0)\}_{|n|>N}$ и $\{\widetilde{\tau}_{2,n}(0)\}_{|n|>N}$ ограничены. Остается оценить
производные:
$$
\widetilde{\tau}'_{j,n}(x)=\omega_{1,n}\eta'_{j1}(x,\lambda_n)+\omega_{2,n}\eta'_{j2}(x,\lambda_n),
$$
а значит, согласно \eqref{stripas2},
\begin{equation*}
|\widetilde{\tau}'_{1,n}(x)|\leqslant M|p_2(x)|(|\omega_{1,n}|+|\omega_{2,n}|), \qquad |\widetilde{\tau}'_{2,n}(x)|\leqslant M|p_3(x)|(|\omega_{1,n}|+|\omega_{2,n}|).
\end{equation*}
Отсюда сразу следует оценка \eqref{tau} для ненормированных функций $\widetilde{\mathbf y}_n$. Так как нормы $\{\|\widetilde{\mathbf y}_n\|_{\mathbb H}\}_{|n|>N}$
отделены от нуля, то эта оценка сохранится и после нормировки.
\end{proof}

\section{Функция Грина.}
Мы покажем, что резольвента оператора $\mathcal L_{P,U}$ компактна и изучим асимптотическое поведение производящей функции $G(t,x,\lambda)$ этого компактного
оператора (функции Грина) при $\lambda\to\infty$.
\begin{proposition}\label{tm:Green} Пусть потенциал $P$ имеет вид \eqref{P}.
Резольвента $\mathfrak R(\lambda)=(\mathcal L_{P, U}-\lambda I)^{-1}$ регулярного оператора $\mathcal L_{P,U}$ определена при всех $\lambda\in\mathbb C\setminus
\{\lambda_n\}_{n\in\mathbb Z}$, где $\lambda_n$
--- собственные значения оператора $\mathcal L_{P,U}$, и является интегральным оператором в $\mathbb H$
\begin{equation}\label{Green}
\mathfrak R(\lambda)\mathbf f=\int_0^\pi G(t,x,\lambda)\mathbf f(t)dt.
\end{equation}
Функция $G(t,x,\lambda)$ непрерывна на квадрате $(t,x)\in[0,\pi]^2$ за исключением диагонали $x=t$.
\end{proposition}
\begin{proof}
Матрица $E(x,\lambda)$, определенная в \eqref{matrE}, удовлетворяет уравнению
$$
BE'(x,\lambda)+P(x)E(x,\lambda)=\lambda E(x,\lambda),
$$
причем $E(0,\lambda)=I$. Применим метод  вариации постоянных к уравнению $\ell_P(\mathbf y)=\lambda\mathbf y+\mathbf f$. Тогда решение этого уравнения примет вид
\begin{equation}\label{Varconst1}
\mathbf y(x,\lambda)=\omega_1\mathbf e_1(x,\lambda)+\omega_2\mathbf e_2(x,\lambda)-\int_x^\pi E(x, \lambda)E^{-1}(t, \lambda)B^{-1}\mathbf f(t)\,dt,
\end{equation}
где $\omega_1$ и $\omega_2$ --- произвольные числа. Легко видеть, что
$$\mathbf y(0,\lambda)=\omega-\int_0^\pi E^{-1}(t, \lambda)B^{-1}\mathbf f(t)\,dt, \quad \mathbf y(\pi,\lambda)=E(\pi,\lambda)\omega,\quad\text{где }\omega=(\omega_1,\omega_2)^t.
$$
Для определения вектора $\omega$ воспользуемся краевыми условиями $ C\mathbf y(0)+D\mathbf y(\pi)=0$, введенными в \eqref{matrU}. Тогда
$$
\omega=\int_0^\pi M^{-1}(\lambda)CE^{-1}(t,\lambda)B^{-1}\mathbf f(t)\,dt,
$$
где $M(\lambda)=C+DE(\pi,\lambda)$. Матрица $M^{-1}(\lambda)$ определена в точности тогда, когда $\Delta(\lambda)=\det M(\lambda)\ne0$, т.е. для всех
$\lambda\in\mathbb C\setminus\{\lambda_n\}_{n\in\mathbb Z}$. Функция $\mathbf y(x,\lambda)$ теперь принимает вид
$$
\mathbf y(x,\lambda)=\int_0^\pi G(t,x,\lambda)\mathbf f(t)\,dt,
$$
где
\begin{equation}\label{Greenf}
G(t,x,\lambda)=\begin{cases}E(x,\lambda)M^{-1}(\lambda)CE^{-1}(t,\lambda)B^{-1},\
\text{при}\ t<x,\\
E(x,\lambda)(M^{-1}(\lambda)C-I)E^{-1}(t,\lambda)B^{-1},\ \text{при}\ t>x,\end{cases}
\end{equation}
что доказывает требуемое утверждение.
\end{proof}
\begin{definition} Функция $G(t,x,\lambda)$ называется {\it функцией Грина}
оператора $\mathcal L_{P, U}$. Через $G_0(t, x, \lambda)$ будем обозначать функцию Грина невозмущенного оператора $\mathcal L_{0, U}$.
\end{definition}
Отметим, что из доказанного утверждения следует компактность в пространстве $\mathbb H$ оператора $\mathfrak R(\lambda)$ при любом
$\lambda\notin\{\lambda_n\}_{n\in\mathbb Z}$ (отсюда, в частности, следует замкнутость оператора $\mathcal L_{P,U}$).

Наша ближайшая цель --- получить оценки для функции $G(t,x,\lambda)$. Эти оценки являются ключевыми для доказательства полноты системы собственных и присоединенных
функций оператора $\mathcal L_{P,U}$. Для упрощения дальнейших выкладок обозначим матрицу $E(x,\lambda)E^{-1}(a,\lambda)$ через $\mathcal E(a,x,\lambda)$, где
$0\leqslant a,\,x\leqslant\pi$, $\lambda\in\mathbb C$, и найдем ее в явном виде. Мы уже отмечали в доказательстве утверждения \ref{tm:0.7}, что $\det
E(x,\lambda)\equiv1$. Тогда
$$
E^{-1}(a,\lambda)=\begin{pmatrix}e_{22}(a,\lambda)&-e_{12}(a,\lambda)\\ -e_{21}(a,\lambda)&e_{11}(a,\lambda)\end{pmatrix},
$$
откуда
$$
\mathcal E(a,x,\lambda)=\begin{pmatrix}\mathcal E_{11}(a,x,\lambda)&\mathcal E_{12}(a,x,\lambda)\\\mathcal E_{21}(a,x,\lambda)&\mathcal
E_{22}(a,x,\lambda)\end{pmatrix},
$$
где
\begin{gather}
\mathcal
E_{j1}(a,x,\lambda)=e_{j1}(x,\lambda)e_{22}(a,\lambda)-e_{j2}(x,\lambda)e_{21}(a,\lambda),\label{calE}\\
\mathcal E_{j2}(a,x,\lambda)=e_{j2}(x,\lambda)e_{11}(a,\lambda)-e_{j1}(x,\lambda)e_{12}(a,\lambda),\quad j=1,\,2.\notag
\end{gather}
В случае $P(x)\equiv0$ будем использовать обозначения $\mathcal E^0(a,x,\lambda)=(\mathcal E^0_{jk}(a,x,\lambda))$, $j,\,k=1,\,2$.
\begin{lemma} \label{lm:0.4}
Матрица $\mathcal{E}(a,x,\lambda)$ удовлетворяет уравнению
\begin{equation}\label{maineq}
B\mathcal{E}'(x)+P(x)\mathcal{E}(x)=\lambda\mathcal{E}(x),\quad x\in[0,\pi]
\end{equation}
и начальному условию $\mathcal{E}(a,a,\lambda)=I$. Функции $\mathcal{E}_{ij}(a,x,\lambda)$ аналитичны по $\lambda$ во всей комплексной плоскости  и при
$\lambda\to\infty$ верно представление
\begin{equation}\label{Eb}
\mathcal E(a,x,\lambda)= \begin{pmatrix}e^{i\lambda\xi}\cdot(1+o(1))+e^{-i\lambda\xi}\cdot o(1)&\quad e^{i\lambda\xi}\cdot o(1)+e^{-i\lambda\xi}\cdot o(1)\\
e^{i\lambda\xi}\cdot o(1)+e^{-i\lambda\xi}\cdot o(1)&\quad e^{i\lambda\xi}\cdot o(1)+e^{-i\lambda\xi}\cdot (1+o(1))
\end{pmatrix},\ \xi=x-a,
\end{equation}
при $\mathbb C\ni\lambda\to\infty$ равномерно по $0\leqslant a,\ x\leqslant\pi$.
\end{lemma}
\begin{proof}
То, что матрица $\mathcal E(a,x,\lambda)$ удовлетворяет уравнению \eqref{maineq} сразу следует из \eqref{calE}. Равенство $\mathcal E(a,a,\lambda)=I$ очевидно.
Асимптотическое представление \eqref{Eb} следует из \eqref{mainasc}.
\end{proof}
Для дальнейшего нам необходимы сведения об операторе $(\mathcal L_{P, U})^*$.
\begin{proposition}\label{lm:0.1}
Сопряженным к регулярному оператору $\mathcal L_{P, U}$ является оператор, который задается сопряженным дифференциальным выражением
$$
\ell_{P^*}(\mathbf y)=B\mathbf y'+P^*(x)\mathbf y,\qquad \text{где} \ P^*(x)=\begin{pmatrix}0&\overline{p_3}(x)\\\overline{p_2}(x)&0\end{pmatrix},
$$
и сопряженными краевыми условиями. Сопряженные краевые условия
выписываются неоднозначно, в частности, матрицу краевых условий
можно взять равной
$$
\mathcal{U}^*=\begin{pmatrix}\overline{J_{23}}&\overline{J_{13}}&-\overline{J_{12}}&0\\
0&-\overline{J_{34}}&\overline{J_{24}}&\overline{J_{23}}\end{pmatrix}.
$$
Для любого регулярного (сильно регулярного) оператора $\mathcal L_{P,U}$ сопряженный оператор $\left(\mathcal L_{P,U}\right)^*=\mathcal L_{P^*, U^*}$
также является регулярным (сильно регулярным). Собственные значения оператора $\mathcal L_{P^*,U^*}$ совпадают (с учетом кратности) с числами
$\overline{\lambda}_n$, где $\lambda_n$, $n\in\mathbb Z$,
--- собственные значения оператора $\mathcal L_{P,U}$.

Для всякого $\lambda\notin\{\overline{\lambda}_n\}_{n\in\mathbb Z}$ определена резольвента $\mathfrak R^*(\lambda)=(\mathcal L_{P^*,U^*}-\lambda I)^{-1}$, которая
имеет вид
$$
\mathfrak R^*(\lambda)\mathbf f=\int_0^\pi G^*(t,x,\lambda)\mathbf f(t)dt.
$$
Матрица $G^*(t,x,\lambda)=(g^*_{jk}(t,x,\lambda))$ связана с функцией $G(t,x,\lambda)=(g_{jk}(t,x,\lambda))$, введенной в \eqref{Green}, соотношениями
\begin{equation}\label{conjGreen}
g^*_{jk}(t,x,\lambda)=\overline{g_{kj}}(x,t,\overline{\lambda}),\qquad j,\,k=1,\,2,\quad x,\,t\in[0,\pi],\ \lambda\in\mathbb
C\setminus\{\overline{\lambda}_n\}_{n\in\mathbb Z}.
\end{equation}
\end{proposition}
\begin{proof}
Вид сопряженного дифференциального выражения следует из леммы \ref{lem:0.1}. Вид сопряженных краевых условий и их регулярность проверяется
непосредственными вычислениями с использованием тождества \eqref{eq:lagr} и определения \ref{def:reg}. Соотношения \eqref{conjGreen} общеизвестны.
\end{proof}
Для сокращения записи далее будем обозначать
\begin{gather*}
U_\delta(\lambda_n)=\{z\in\mathbb C: |z-\lambda_n|<\delta\},\qquad \Omega_{\delta}=\mathbb C\setminus\bigcup\limits_{n\in\mathbb Z}U_\delta(\lambda_n),\\
\Omega_{\alpha,\delta}=\Pi_\alpha\cap\Omega_{\delta}, \qquad\Omega_{\alpha,\delta,R}=\{z\in\Omega_{\alpha,\delta}: |\mathrm{Re\,} z|>R\}.
\end{gather*}
\begin{lemma}\label{Deltaest}
Для любого $\delta>0$ существует такое число $M=M(P,U,\delta)$, что при всех $\lambda\in\overline{\Omega_\delta}$, характеристический определитель регулярного
оператора $\mathcal L_{P,U}$ удовлетворяет оценке $\left|\Delta(\lambda)\right|\geqslant Me^{\pi|\mathrm{Im\,}\lambda|}$.
\end{lemma}
\begin{proof}
Согласно теореме \ref{tm:0.8} и утверждению \ref{tm:0.7}, найдется такое число $\alpha_0>0$, что все круги $U_\delta(\lambda_n)$ лежат в полосе $\Pi_{\alpha_0}$ и при
$\lambda\to\infty$ вне $\Pi_{\alpha_0}$ справедливо равенство $\Delta(\lambda)/\Delta_0(\lambda)=1+o(1)$. Увеличивая, если нужно, число $\alpha_0$, можно считать, что
$|\Delta(\lambda)|\geqslant |\Delta_0(\lambda)|/2$ для всех $\lambda\notin\Pi_{\alpha_0}$. Тогда из неравенства \eqref{Delta0est} следует доказываемое неравенство при
$\mathrm{Im\,}\lambda>\alpha_0$. Случай $\mathrm{Im\,}\lambda<-\alpha_0$ аналогичен. Для завершения доказательства остается показать, что для всех точек
$\lambda\in\overline{\Omega_{\alpha_0,\delta}}$ справедлива оценка $|\Delta(\lambda)|\geqslant M$ при некотором $M>0$. Согласно теореме \ref{tm:0.8}, найдется такое
число $R$, что для всех собственных значений $\lambda_n$, $|\lambda_n|>R$, справедливы неравенства $|\lambda_n-\lambda_n^0|<\delta/2$. В силу периодичности функции
$\Delta_0(\lambda)$, существует такое $m>0$, что $|\Delta_0(\lambda)|\geqslant m$ в $\Pi_{\alpha_0}$ вне кругов $U_{\delta/2}(\lambda_n^0)$. Поскольку
$\Delta(\lambda)=\Delta_0(\lambda)+o(1)$ при $\lambda\to\infty$ в полосе $\lambda\in\Pi_{\alpha_0}$ (см. утверждение \ref{tm:0.7}), то, увеличивая, если необходимо,
число $R$, можно считать, что при $|\mathrm{Re\,}\lambda|>R$ выполнена оценка $|\Delta(\lambda)|\geqslant m/2$ в $\Pi_{\alpha_0}$ вне кругов
$U_{\delta/2}(\lambda_n^0)$. Так как при $|\mathrm{Re\,}\lambda|>R$ круг $U_{\delta/2}(\lambda_n^0)$ содержится в круге $U_\delta(\lambda_n)$, то оценка
$|\Delta(\lambda)|\geqslant m/2$ выполнена при всех $\lambda\in\overline{\Omega}_{\alpha_0,\delta,R}$. Наконец, на компакте
$$
\{\lambda:|\mathrm{Im\,}\lambda|\leqslant\alpha_0,\,|\mathrm{Re\,}\lambda|\leqslant R,\,|\lambda-\lambda_n|\geqslant\delta,\,n\in\mathbb Z\}
$$
функция $\Delta(\lambda)$ не обращается в ноль, а значит, отделена от нуля.
\end{proof}
\begin{theorem}\label{tm:Greenest}
Пусть $\mathcal L_{P,U}$ --- произвольный оператор Дирака с потенциалом вида \eqref{P} и регулярными краевыми условиями $U$. Для любого $\delta>0$ существует такое
число $M=M(P,U,\delta)$, что в $\overline{\Omega_{\delta}}$ функция $G(t,x,\lambda)=(g_{jk}(t,x,\lambda))$ оператора $\mathcal L_{P,U}$ удовлетворяет оценке
$$
|g_{jk}(t,x,\lambda)|\leqslant M.
$$
Кроме того, для любых положительных чисел $\alpha$, $\delta$ и $\varepsilon$ найдется такое $R>0$, что при всех $\lambda\in\overline{\Omega_{\alpha,\delta,R}}$ и всех
$t,x\in [0,\pi]$ выполнено
$$
|g_{jk}(t, x, \lambda)-g_{jk}^0(t, x, \lambda)|<\varepsilon,
$$
где $G_0(t,x,\lambda)=(g_{jk}^0(t,x,\lambda))$ --- функция Грина оператора $\mathcal L_{0,U}$.
\end{theorem}
\begin{proof}
Матрицы $E(x,\lambda)$, $M(\lambda)$, $C$ и $E^{-1}(t,\lambda)$ явно выписаны в \eqref{matrE}, \eqref{matrM}, \eqref{matrU} и \eqref{calE} соответственно. Тогда из
\eqref{Greenf} непосредственными вычислениями получаем, что
\begin{multline}\label{Greensuper}
G(t, x, \lambda)=i\left(\frac{J_{12}}{\Delta(\lambda)}-\chi_{t>x}(t,x)\right)\begin{pmatrix} \mathcal E_{11}(t, x, \lambda)&-\mathcal E_{12}(t, x, \lambda) \\
\mathcal
E_{21}(t, x, \lambda)&-\mathcal E_{22}(t, x, \lambda)\end{pmatrix}+\\
+\frac{i}{\Delta(\lambda)}\begin{pmatrix} \mathcal E_{11}(\pi, x, \lambda)&-\mathcal E_{12}(\pi, x, \lambda) \\ \mathcal E_{21}(\pi, x, \lambda)&-\mathcal E_{22}(\pi,
x,
\lambda)\end{pmatrix}\cdot\begin{pmatrix}J_{14}&J_{24}\\
J_{13}&J_{23}\end{pmatrix}\cdot\begin{pmatrix}e_{22}(t,\lambda)&e_{12}(t,\lambda)\\
-e_{21}(t,\lambda)&-e_{11}(t,\lambda)\end{pmatrix},
\end{multline}
где $\chi_{t>x}$ --- характеристическая функция треугольника $t>x$, а $\Delta(\lambda)$ --- определитель, введенный в определении \ref{def.3}. Пусть $\alpha_0>0$
таково, что полоса $\Pi_{\alpha_0}$ содержит все круги $U_\delta(\lambda_n)$. Пусть вначале $|\mathrm{Im\,}\lambda|\geqslant\alpha_0$. Проведем оценку функции $G(t,
x, \lambda)$ на треугольнике $0\leqslant t<x\leqslant\pi$. В силу представлений \eqref{Eb}, функции $\mathcal{E}_{jk}(t,x,\lambda)$ удовлетворяют оценкам
$$
|\mathcal{E}_{jk}(t,x,\lambda)|\leqslant Me^{\left|\mathrm{Im\,}\lambda\right|(x-t)}.
$$
Аналогично,
$$
|\mathcal{E}_{jk}(\pi,x,\lambda)|\leqslant Me^{\left|\mathrm{Im\,}\lambda\right|(\pi-x)},\qquad \text{а}\ \  |e_{jk}(t,\lambda)|\leqslant
Me^{\left|\mathrm{Im\,}\lambda\right|t},
$$
где $M$ не зависит от $t$, $x$ и $\lambda$. Применяя лемму \ref{Deltaest}, видим, что вне полосы $\Pi_{\alpha_0}$,
$$
|g_{jk}(t,x,\lambda)|\leqslant M\left(e^{\left|\mathrm{Im\,}\lambda\right|(x-t-\pi)}+e^{\left|\mathrm{Im\,}\lambda\right|(t-x)}\right)\leqslant 2M,
$$
поскольку оба числа $x-t-\pi$ и $t-x$ неположительны. Для оценки функции $G(t,x,\lambda)$ на треугольнике $0\leqslant x<t\leqslant\pi$ воспользуемся соотношением
\eqref{conjGreen}, согласно которому $|g_{jk}(t,x,\lambda)|=|g^*_{kj}(x,t,\overline{\lambda})|$. Поскольку координаты точки $x$ и $t$ поменялись местами, а
$\overline{\lambda}$ по-прежнему лежит вне полосы $\Pi_{\alpha_0}$, то мы можем применить рассуждения, приведенные выше, к функциям $g^*_{kj}$ и вне полосы
$\Pi_{\alpha_0}$ оценка функции $G$ получена. Согласно асимптотическим представлениям \eqref{stripas} и \eqref{Eb}, функции $e_{jk}$ и $\mathcal E_{jk}$ ограничены в
произвольной полосе $\Pi_\alpha$. Отсюда и из леммы \ref{Deltaest} следует ограниченность функции $G(t,x,\lambda)$ в $\overline{\Omega}_{\alpha_0,\delta}$.

Докажем второе утверждение теоремы. Зафисксируем числа $\alpha$ и $\delta$. Из теоремы \ref{B} следует, что
$$
e_{jk}(x,\lambda)=e_{jk}^0(x,\lambda)+o(1)\quad \text{при}\ \  \Pi_\alpha\ni\lambda\to\infty
$$
равномерно по $x$, а из \eqref{Eb} следует, что и
$$
\mathcal E_{jk}(a,x,\lambda)=\mathcal E^0_{jk}(a,x,\lambda)+o(1)\quad \text{при}\ \  \Pi_\alpha\ni\lambda\to\infty
$$
равномерно по $x$ и $a$. Согласно лемме \ref{Deltaest}, найдутся такие положительные числа $R$ и $M$, что при всех $\lambda\in\overline{\Omega_{\alpha,\delta,R}}$
выполнены неравенства $|\Delta(\lambda)|\geqslant M$ и $|\Delta_0(\lambda)|\geqslant M$. Тогда из утверждения \ref{tm:0.7} следует
$$
\Delta^{-1}(\lambda)=\Delta_0^{-1}(\lambda)+o(1)\quad \text{при}\ \ \overline{\Omega_{\alpha,\delta,R}}\ni\lambda\to\infty.
$$
Подставляя эти асимптотические представления в равенство \eqref{Greensuper} (записанное для $G(t,x,\lambda)$ и $G_0(t,x,\lambda)$) и учитывая равномерную
ограниченность функций $\Delta^{-1}(\lambda)$, $e^0_{jk}(x,\lambda)$ и $\mathcal E^0_{jk}(t,x,\lambda)$ на множестве $\lambda\in\overline{\Omega_{\alpha,\delta,R}}$,
$x,\,t\in[0,\pi]$, получаем необходимую оценку.
\end{proof}
Факты, которые мы сформулируем ниже, хорошо известны и основаны на канонической работе \cite{Ke} (см. также \cite[Гл. 1]{Na}).
\begin{definition}
Система функций $\mathbf y^{j,1}$, $\mathbf y^{j,2}$, $\dots$, $\mathbf y^{j,m}$ называется {\it цепочкой функций, присоединенных к собственной
функции} $\mathbf y^j$ оператора $\mathcal L_{P,U}$ с собственным значением $\lambda_0$, если все они лежат в области определения $\mathfrak
D(\mathcal L_{P,U})$ и удовлетворяют системе уравнений $\mathcal L_{P,U}\mathbf y^{j,q}=\lambda_0\mathbf y^{j,q}+\mathbf y^{j,q-1}$,
$q=1,\,\dots,\,m$ (здесь и далее $\mathbf y^{j,0}=\mathbf y^j$ --- собственные функции). Будем говорить, что собственная функция $\mathbf y^j$ {\it
имеет кратность} $m_0$, если существует цепочка из присоединенных к ней функций длины $m_0-1$, но не существует такой цепочки длины $m_0$. Пусть $p$
--- размерность собственного подпространства $\mathcal H_0$, отвечающего собственному значению $\lambda_0$. Обозначим через $\mathbf y^1\in\mathcal
H_0$ собственную функцию, имеющую максимальную кратность, через $\mathbf y^2\in\mathcal H_0$ --- собственную функцию максимальной кратности, линейно
независимую с $\mathbf y^1$ и т.д. Пусть $m_j$ --- кратность собственной функции $\mathbf y^j$, а $\mathbf y^{j,k}$, $k=1,\,\dots,\,m_j-1$
--- соответствующие присоединенные функции. Система
$\{\mathbf y^{j,k}\}$, где $1\leqslant j\leqslant p$, а $0\leqslant k\leqslant m_j-1$, называется {\it канонической системой} собственных и
присоединенных функций оператора $\mathcal L_{P,U}$, отвечающей собственному значению $\lambda_0$.
\end{definition}
Легко видеть, что любая каноническая система $\{\mathbf y^{j,k}\}$ образуют базис в собственном подпространстве, отвечающем собственному значению $\lambda_0$. Следуя
работе \cite{Ke}, обозначим через $\mathbf y\mathbf z$ оператор в пространстве $\mathbb H$, действующий по правилу $\mathbf f\mapsto\langle\mathbf f,\mathbf
z\rangle\mathbf y$.

\begin{THEOREM}\label{D}
Для любого собственного значения $\lambda_0$ регулярного оператора $\mathcal L_{P,U}$ размерность $p$ собственного подпространства не превосходит $2$. Кратность нуля
функции $\Delta(\lambda)$ в точке $\lambda_0$ совпадает с суммой $m_1+m_2$ (в случае $p=1$ полагаем $m_2=0$). При этом функция $G(t,x,\lambda)$ имеет полюс порядка
$m_1$ в точке $\lambda_0$. Пусть $\{\mathbf y^{j,k}\}$
--- произвольная каноническая система сосбтвенных и присоединенных
функций оператора $\mathcal L_{P,U}$, отвечающая собственному значению $\lambda_0$. Тогда найдется такая каноническая система $\{\mathbf z^{j,k}\}$ собственных и
присоединенных функций сопряженного оператора $\left(\mathcal L_{P,U}\right)^*$, отвечающая собственному значению $\overline{\lambda_0}$, что главная часть ряда
Лорана резольвенты $\mathfrak R(\lambda)=(\mathcal L_{P,U}-\lambda I)^{-1}$ в точке $\lambda_0$ будет иметь вид
\begin{multline}\label{biortog}
\frac{\mathbf y^{1,0}\overline{\mathbf z^{1,0}}}{(\lambda-\lambda_0)^{m_1}}+ \frac{\mathbf y^{1,0}\overline{\mathbf z^{1,1}}+\mathbf y^{1,1}\overline{\mathbf
z^{1,0}}}{(\lambda-\lambda_0)^{m_1-1}}+\dots
+\frac{\mathbf y^{1,0}\overline{\mathbf z^{1,m_1-1}}+\dots+\mathbf y^{1,m_1-1}\overline{\mathbf z^{1,0}}}{\lambda-\lambda_0}+\\
+\frac{\mathbf y^{2,0}\overline{\mathbf z^{2,0}}}{(\lambda-\lambda_0)^{m_2}}+ \frac{\mathbf y^{2,0}\overline{\mathbf z^{2,1}}+\mathbf y^{2,1}\overline{\mathbf
z^{2,0}}}{(\lambda-\lambda_0)^{m_2-1}}+\dots +\frac{\mathbf y^{2,0}\overline{\mathbf z^{2,m_2-1}}+\dots+\mathbf y^{2,m_2-1}\overline{\mathbf
z^{2,0}}}{\lambda-\lambda_0}.
\end{multline}
\end{THEOREM}
\begin{definition}\label{def:EAF}
Для каждого собственного значения $\lambda_0$ регулярного оператора $\mathcal L_{P,U}$ выберем произвольную каноническую систему $\{\mathbf
y^{j,k}\}$ собственных и присоединенных функций с тем лишь условием, что собственные функции этой системы имеют единичную норму. В силу теоремы
\ref{D}, количество векторов в системе $\{\mathbf y^{j,k}\}$ совпадает с порядклм нуля $\lambda_0$ функции $\Delta(\lambda)$. Занумеруем векторы этой
системы (в порядке $\mathbf y^1,\ \mathbf y^{1,1},\,\dots,\,\mathbf y^{1,m_1-1},\ \mathbf y^2,\ \mathbf y^{2,1},\,\dots,\,\mathbf y^{2,m_2-1}$)
индексами $n\in\mathbb Z$ в соответствии с нумерацией собственных значений. {\it Системой собственных и присоединенных функций} $\{\mathbf
y_n\}_{n\in\mathbb Z}$ оператора $\mathcal L_{P,U}$ мы будем называть объединение всех канонических систем. Оператор
$$
\mathcal P_{\lambda_0}=\mathbf y^{1,0}\overline{\mathbf z^{1,m_1-1}}+\dots+\mathbf y^{1,m_1-1}\overline{\mathbf z^{1,0}}+\mathbf y^{2,0}\overline{\mathbf
z^{2,m_2-1}}+\dots+\mathbf y^{2,m_2-1}\overline{\mathbf z^{2,0}}
$$
называется {\it спектральным проектором на корневое подпространство}, отвечающее собственному значению $\lambda_0$.
\end{definition}
Полученное в теореме \ref{tm:Greenest} асимптотическое представление для функции Грина позволяет нам получить асимптотические формулы для спектральных проекторов в
случае регулярных, но не сильно регулярных краевых условий. В этом случае все собственные значения оператора $\mathcal L_{0,U}$ двукратны (см. утверждение
\ref{tm:0.3}), а именно $\lambda^0_{2n}=\lambda^0_{2n+1}$, $n\in\mathbb Z$. Поскольку $\lambda_n=\lambda_n^0+o(1)$, то $|\lambda_{2n}-\lambda_{2n+1}|\to0$ при
$n\to\pm\infty$.
\begin{definition}\label{def:RP} Выберем число $N_0$ так, что для всех $n$, $|n|\geqslant N_0$ выполнено $|\lambda_{2n}-\lambda_{2n}^0|<1/8$ и
$|\lambda_{2n+1}-\lambda_{2n}^0|<1/8$. Обозначим
\begin{equation}\label{RiszPr}
\mathcal P_n:=\frac1{2\pi i}\int_{|\lambda-\lambda_{2n}^0|=1/4}\mathfrak R(\lambda)d\lambda,\quad n=\pm N_0,\,\pm(N_0+1),\dots,\quad\text{где }\mathfrak
R(\lambda)=(\mathcal L_{P,U}-\lambda I)^{-1}.
\end{equation}
Из представления \eqref{biortog} следует, что $\mathcal P_n$ является спектральным проектором на корневое подпространство, отвечающее собственным значениям
$\lambda_{2n}$ и $\lambda_{2n+1}$, которое мы обозначим $\mathcal H_n$. Определим также операторы
$$
\mathcal P_n^0:=\frac1{2\pi i}\int_{|\lambda-\lambda_{2n}^0|=1/4}\mathfrak R_0(\lambda)d\lambda,\quad n=\pm N_0,\,\pm(N_0+1),\dots,\quad\text{где }\mathfrak
R_0(\lambda)=(\mathcal L_{0,U}-\lambda I)^{-1}.
$$
--- спектральные проекторы на корневые подпространства оператора $\mathcal L_{0,U}$, отвечающие собственным значениям $\lambda_{2n}^0=\lambda_{2n+1}^0$.
\end{definition}
Заметим, что оператор $\mathfrak R(\lambda): \mathbf f\mapsto\int_0^\pi G(t,x,\lambda)\mathbf f(t)dt$ корректно определен при $\lambda\ne\lambda_n$ не только как
оператор в пространстве $\mathbb H$, но и как оператор из $L_1[0,\pi]$ в $C[0,\pi]$. То же справедливо и для операторов $\mathcal P_n$ и $\mathcal P_n^0$. В следующей
теореме мы оценим норму их разности именно как операторов из $L_1[0,\pi]$ в $C[0,\pi]$.
\begin{theorem}\label{tm:Sppr}
Для любого регулярного, но не сильно регулярного оператора $\mathcal L_{P,U}$:
$$
\|\mathcal P_n-\mathcal P_n^0\|_{L_1\to C}\longrightarrow0\quad \text{при}\ \ |n|\to\infty.
$$
\end{theorem}
\begin{proof}
Легко видеть, что при $|n|\geqslant N_0$
$$
\|\mathcal P_n-\mathcal P_n^0\|_{L_1\to C}\leqslant \frac14\max_{|\lambda-\lambda_{2n}^0|=1/4}\ \max_{j,k\in\{1,\,2\}}\
\sup_{t,x\in[0,\pi]}|g_{jk}(t,x,\lambda)-g_{jk}^0(t,x,\lambda)|.
$$
Теперь утверждение теоремы следует из теоремы \ref{tm:Greenest}.
\end{proof}

\vskip 0.5cm

\section{Минимальность, полнота и базисность Рисса.}
Нашей следующей задачей является доказательство полноты и минимальности системы собственных и присоединенных функций регулярного оператора $\mathcal L_{P,U}$. Мы
проведем это доказательство классическим способом, причем ключевую роль будет играть оценка, полученная в теореме \ref{tm:Greenest}. Напомним, что система $\{x_n\}$
векторов банахова пространства $H$ называется \textit{полной}, если ее линейная оболочка плотна в $H$. Система называется \textit{минимальной}, если при удалении
произвольного вектора $x_k$ из системы свойство полноты теряется.
\begin{theorem}\label{tm:compl}
Пусть потенциал $P$ имеет вид \eqref{P}, а краевые условия \eqref{matrU} регулярны. Тогда система $\{\mathbf y_n\}_{n\in\mathbb Z}$ собственных и присоединенных
функций оператора $\mathcal L_{P,U}$ (см. определение \ref{def:EAF}) полна и минимальна в пространстве $\mathbb H$.
\end{theorem}
\begin{proof}
Вначале докажем полноту системы $\{\mathbf y_n\}_{n\in\mathbb Z}$. Пусть функция $\mathbf f\in\mathbb H$ ортогональна всем векторам этой системы. Зафиксируем
произвольный вектор $\mathbf g\in\mathbb H$ и рассмотрим функцию $\Phi(\lambda):=\langle \mathfrak R^*(\lambda)\mathbf f,\mathbf g\rangle$, определенную в области
$\mathbb C\setminus\{\overline{\lambda_n}\}_{n\in\mathbb Z}$. Согласно \eqref{biortog}, эта функция имеет устранимые особенности в точках $\overline{\lambda_n}$, т.е.
после доопределения в них, явлется целой. В силу теоремы \ref{tm:Greenest}, для любого $\delta>0$ в области $\mathbb C\setminus\bigcup_{n\in\mathbb
Z}U_\delta(\overline{\lambda_n})$ справедлива оценка
$$
|\Phi(\lambda)|\leqslant M\|\mathbf f\|_{\mathbb H}\|\mathbf g\|_{\mathbb H},
$$
где $M$ не зависит от $\lambda$. Заметим, что для случая сильно регулярных краевых условий
$$
\inf_{n\ne m}|\overline{\lambda_n^0}-\overline{\lambda_m^0}|=d>0
$$
(см. утверждение \ref{tm:0.3}). В регулярном, но не сильно регулярном случае,
$$
\inf_{|n-m|\geqslant2}|\overline{\lambda_n^0}-\overline{\lambda_m^0}|=d>0.
$$
Выберем число $\delta$ равным $d/4$ --- тогда круги $U_\delta(\overline{\lambda_n^0})$ либо не пересекаются, либо разбиваются на пары, не пересекающиеся между собой.
Этим же свойством, очевидно, обладают и круги $U_\delta(\overline{\lambda_n})$ для всех $n$ таких, что $|\overline{\lambda_n}-\overline{\lambda_n^0}|<d/4$. В силу
теоремы \ref{tm:0.8}, последнее неравенство выполнено при $|n|\geqslant N$ для некоторого $N$. Таким образом, вне некоторого круга $\{|z|\leqslant R\}$ множество
точек $\lambda$, для которых неравенство $|\Phi(\lambda)|\leqslant M\|\mathbf f\|_{\mathbb H}\|\mathbf g\|_{\mathbb H}$ еще не доказано, представляет собой счетное
объединение ограниченных непересекающихся областей. По принципу максимума, это неравенство будет справедливо в каждой из данных областей, значит, и всюду в области
$\{|z|>R\}$, а следовательно, и во всей комплексной плоскости. Из теоремы Лиувилля следует, что функция $\Phi(\lambda)$ является постоянной. Тогда функция
$$
\Phi'(\lambda)=\langle(\mathfrak R^*(\lambda))'\mathbf f, \mathbf g\rangle=\langle(\mathfrak R^*(\lambda))^2\mathbf f, \mathbf g\rangle\equiv 0.
$$
Поскольку функция $\mathbf g$ выбиралась произвольной, то $(\mathfrak R^*(\lambda))^2\mathbf f\equiv 0$, откуда $\mathbf f=0$. Полнота системы $\{\mathbf
y_n\}_{n\in\mathbb Z}$ доказана.

Для доказательства минимальности системы $\{\mathbf y_n\}_{n\in\mathbb Z}$ достаточно доказать существование биортогональной системы. Мы построим ее на базе системы
$\{\mathbf z_n\}$, полученной объединением всех канонических систем $\{\mathbf z^{j,k}\}$, определенных в разложении \eqref{biortog} (т.е. системы собственных и
присоединенных функций оператора $\left(\mathcal L_{P,U}\right)^*$). Рассмотрим некоторое фиксированное собственное значение $\lambda$ оператора $\mathcal L_{P,U}$
алгебраической кратности $p$ и обозначим соответствующее корневое подпространство через $\mathcal H_\lambda$. Корневое подпространство, отвечающее собственному
значению $\overline{\lambda}$ оператора $\left(\mathcal L_{P,U}\right)^*$ обозначим $\mathcal H_\lambda^*$. Прежде всего заметим, что если $\mathcal L_{P,U}\mathbf
y=\lambda\mathbf y$, $\left(\mathcal L_{P,U}\right)^*\mathbf z=\mu\mathbf z$ и $\mu\ne\overline{\lambda}$, то $\mathbf y\perp\mathbf z$, т.е. $\mathcal H_\lambda\perp
\mathcal H^*_\mu$ при $\lambda\ne \overline{\mu}$. Таким образом, для построения биортогональной системы, достаточно в каждом пространстве $\mathcal H_\lambda^*$
построить базис $\{\mathbf w^{j,k}\}$, биортогональный системе $\{\mathbf y^{j,k}\}$. Нам не потребуется явное представление векторов $\mathbf w^{j,k}$, так что мы
ограничимся доказательством существования такого базиса. Представим этот базис в виде линенйых комбинаций системы $\{\mathbf z^{j,k}\}$, определенной в
\eqref{biortog}. Записав условия биортогональности, получим систему линейных уравнений с матрицей Грама $\big(\langle\mathbf y^{j,k},\mathbf z^{l,m}\rangle\big)$.
Разрешимость системы равносильна невырожденности данной матрицы. Если же матрица вырождена, то найдется ненулевой вектор $\sum c_{j,k}\mathbf z^{j,k}$, ортогональный
всем функциям $\mathbf y^{j,k}$, а значит и вообще всей системе собственных и присоединенных функций оператора $\mathcal L_{P,U}$. Это противоречит полноте данной
системы. Минимальность системы $\{\mathbf y_n\}_{n\in\mathbb Z}$ доказана.
\end{proof}
\begin{definition}\label{def:biortog}
Объединение всех систем $\{\mathbf w^{j,k}\}$ будем называть {\it биортогональной системой} и обозначать $\{\mathbf w_n\}_{n\in\mathbb Z}$. При этом нумерацию мы
ведем так,
что $\langle\mathbf y_n,\mathbf w_m\rangle=\delta_{nm}$.
\end{definition}
Мы переходим к доказательству базисности системы $\{\mathbf y_n\}_{n\in\mathbb Z}$ собственных и присоединенных функций оператора $\mathcal L_{P,U}$.
Напомним (см. \cite[Гл.~6]{GK}), что система \(\{y_n\}_1^\infty\) в гильбертовом пространстве $H$ называется \textit{базисом Рисса}, если существует ограниченный и
ограниченно обратимый оператор \(A\) такой, что система \(\{Ay_n\}_1^\infty\) является ортонормированным базисом в $H$. Напомним еще, что система $\{x_n\}_1^\infty$
элементов гильбертова пространства $H$ называется \textit{бесселевой}, если существует $c>0$ такое, что для любого $x\in H$: $\sum_n|(x,x_n)|^2\leqslant c\|x\|^2$. Мы
докажем, что для любого сильно регулярного оператора $\mathcal L_{P,U}$ система $\{\mathbf y_n\}_{n\in\mathbb Z}$ является базисом Рисса в $\mathbb H$. Отметим, что
этот факт не является простым. Так, например, он не следует из асимптотических формул \eqref{efas1}. Краткое доказательство базисности Рисса для этого случая было
приведено в недавней работе \cite{SavSh14}. Мы проведем здесь подробное доказательство, основываясь на теореме Бари, причем основную роль будут играть представление
\eqref{efas2} и лемма \ref{lem:harm}, приведенная ниже.

\begin{THEOREM}[Н.~К.~Бари]\label{tm:bari} Пусть система $\{\mathbf y_n\}$ гильбертова пространства $H$ полна и минимальна, равно как и
биортогональная к ней система $\{\mathbf z_n\}$. Если обе эти системы обладают свойством бесселевости, то они являются базисами Рисса в $H$.
\end{THEOREM}
Напомним, что  \textit{пространством Харди} $H_2(\mathbb C_+)$ называется пространство аналитических в верхней полуплоскости функций, для которых норма
$$
\|F\|_{H_2}=\sup_{y>0}\left(\int_{\mathbb R}|F(x+iy)|^2dx\right)^{1/2}<\infty.
$$
Для доказательства следующей леммы нам потребуется теорема Карлесона (см., например, \cite[теорема II.3.9]{Gar}).
\begin{THEOREM}[Л.~Карлесон]\label{tm:Carl}
Пусть $\sigma$ --- мера Карлесона в верхней полуплоскости, т.е. для любого квадрата $Q_{a,h}=\{z:\mathrm{Re\,} z\in(a,a+h),\ \mathrm{Im\,} z\in(0,h)\}$ мера
$\sigma(Q_{a,h})$ конечна и $\sigma(Q_{a,h})\leqslant \gamma h$ для некоторого $\gamma>0$. Тогда
$$
\forall f\in H_2(\mathbb C_+):\quad\int|f|^2\,d\sigma\leqslant C\|f\|^2_{H_2},\quad \text{где}\ \  C=C(\gamma).
$$
\end{THEOREM}
\begin{lemma}\label{lem:harm} Пусть $\{\lambda_n\}_{n\in\mathbb Z}$ --- последовательность собственных значений оператора $\mathcal L_{P,U}$  с потенциалом $P(\cdot)\in L_1[0,\pi]$ и
регулярными краевыми условиями $U$. Тогда для всех $f\in L_2[0,\pi]$ справедлива оценка
$$
\sum_{n\in\mathbb Z} \left|\int_0^\pi f(x) e^{i\lambda_n x}\,dx\right|^{2} \leqslant C \|f\|^2_{L_2},\quad\text{где}\ \ C=C(P,U).
$$
\end{lemma}
\begin{proof}
Напомним, что все собственные значения оператора $\mathcal L_{P,U}$ лежат в полосе $\Pi_\alpha$ для некоторого $\alpha=\alpha(P,U)>0$. Рассмотрим целую функцию
$$
F(z)=\int_0^\pi f(x)e^{(iz+\alpha+1) x}\,dx.
$$
Из теоремы Пэли--Винера следует, что функция $F$ принадлежит пространству Харди $H_{2}(\mathbb{C}_+)$ в верхней полуплоскости, причем $\|F\|_{H_{2}}\leqslant
C\|f\|_{L_2}$. Положим $z_n=\lambda_n+i\alpha+i$ и заметим, что
$$
F(z_n)=\int_0^\pi f(x)e^{i\lambda_n x}\,dx.
$$
Пусть теперь $\mu(Q_{a,h})$
--- количество точек $z_n$ (с учетом кратности), лежащих внутри квадрата
$$
Q_{a,h}=\{z:\mathrm{Re\,} z\in(a,a+h),\ \mathrm{Im\,} z\in(0,h)\}.
$$
Легко видеть, что $\mu(Q_{a,h})=0$ при $h<1$. Поскольку $\lambda_n=n+\varkappa_n+o(1)$,  где $\varkappa_n$ зависят только от четности номера $n$, то функция
$s(h)=\sup_{a\in\mathbb R}\mu(Q_{a,h})$ конечна для любого $h$ и $s(h)\sim h$ при $h\to\infty$. Тогда найдется число $\gamma>0$ такое, что $s(h)\leqslant\gamma h$ при
всех $h>0$. Применив теорему \ref{tm:Carl} с мерой $\sigma=\sum_{n\in\mathbb Z}\delta_{z_n}$, получим оценку
$$
\|\{F(z_n)\}_{n\in\mathbb Z}\|_{l_2}\leqslant C(\gamma)\|F\|_{H_2},
$$
что и влечет утверждение леммы.
\end{proof}
Напомним, что все собственные векторы системы $\{\mathbf y_n\}_{n\in\mathbb Z}$ нормированы. Поскольку все собственные значения сильно регулярного оператора $\mathcal
L_{P,U}$ просты, начиная с некоторого номера $N$, то при всех $|n|>N$:  $\|\mathbf y_n\|=1$. Спектр оператора $(\mathcal L_{P,U})^*$ совпадает с множеством
$\{\overline{\lambda}_n\}_{n\in\mathbb Z}$ с совпадением кратностей, а значит, при $|n|>N$, все векторы $\mathbf w_n$ биортогональной системы также являются
собственными для оператора $(\mathcal L_{P,U})^*$. Однако, в отличие от $\mathbf y_n$, они уже могут иметь неединичную норму.
\begin{lemma}\label{lem.unimin}
Пусть $\mathcal L_{P,U}$ --- произвольный сильно регулярный оператор Дирака, а $\{\mathbf w_n\}_{n\in\mathbb Z}$ --- система, биортогональная к $\{\mathbf
y_n\}_{n\in\mathbb Z}$ (см. определение \ref{def:biortog}). Тогда последовательность $\{\|\mathbf w_n\|\}_{n\in\mathbb Z}$ ограничена.
\end{lemma}
\begin{proof}
Обозначим $\widetilde{\mathbf w}_n=\tfrac{\mathbf w_n}{\|\mathbf w_n\|}$ и заметим, что $1=\langle\mathbf y_n,\mathbf w_n\rangle=\|\mathbf w_n\|\langle\mathbf
y_n,\widetilde{\mathbf w}_n\rangle$, а значит $\langle\mathbf y_n,\mathbf w_n\rangle\ne0$, $n\in\mathbb Z$. Кроме того, неравенство $\|\mathbf w_n\|<C$ равносильно
неравенству $\langle\mathbf y_n,\widetilde{\mathbf w}_n\rangle>1/C$. Пусть $\mathbf y_n^0$
--- нормированные собственные функции оператора $\mathcal L_{0,U}$, а $\mathbf w_n^0$
--- нормированные собственные функции оператора $\mathcal L_{0,U^*}$.
Используя \eqref{eigenfunc}, получим
$$
\mathbf y_n^0=\big(\omega_{1,j}\,e^{i\lambda_n^0x},\,\omega_{2,j}\,e^{-i\lambda_n^0x}\big)^t,\qquad \mathbf
w_n^0=\big(\omega^*_{1,j}\,e^{i\overline{\lambda_n^0}x},\,\omega^*_{2,j}\,e^{-i\overline{\lambda_n^0}x}\big)^t,
$$
где $j=0$ при четном $n$ и $j=1$ при нечетном $n$. Тогда
$$
\langle\mathbf y^0_n,\mathbf w^0_n\rangle=\pi(\omega_{1,j}\overline{\omega^*_{1,j}}+\omega_{2,j}\overline{\omega^*_{2,j}}),
$$
т.е. скалярные произведения $\langle\mathbf y_n^0,\mathbf w_n^0\rangle$ зависят только от четности индекса $n$. Поскольку по определению $\langle\mathbf y_n^0,\mathbf
w_n^0\rangle\ne0$, то $|\langle\mathbf y_n^0,\mathbf w_n^0\rangle|\geqslant C>0$ при всех $n\in\mathbb Z$. Все векторы $\mathbf y_n$ и $\mathbf w_n$ при достаточно
больших $|n|$ являются собственными. Из теоремы \ref{tm:0.9} имеем $\langle\mathbf y_n,\widetilde{\mathbf w}_n\rangle=\langle\mathbf y_n^0,\mathbf w_n^0\rangle+o(1)$,
т.е. числа $\left|\langle\mathbf y_n,\widetilde{\mathbf w}_n\rangle\right|$ отделены от нуля при достаточно больших (а значит и при всех) $n$.
\end{proof}
Нам потребуется еще одно несложное утверждение. Оно, однако,
является ключевым для доказательства теоремы \ref{tm.Rieszbas}.
\begin{lemma}\label{lem:mp}
Пусть система $\{\varphi_n(x)\}_1^\infty$ является бесселевой в
пространстве $L_2[a,b]$, а $\{\tau_n(x)\}_1^\infty$ --- абсолютно
непрерывные на $[a,b]$ функции, причем
\begin{equation}
|\tau_n(a)|\leqslant T,\qquad |\tau_n'(x)|\leqslant\tau(x)\in L_1[a,b],\quad n=1,2,\dots,
\end{equation}
где число $T$ и функция $\tau$ не зависят от $n$. Тогда система $\{\varphi_n(x)\tau_n(x)\}_1^\infty$ также является бесселевой в пространстве $L_2[a,b]$.
\end{lemma}
\begin{proof}
Поскольку
$$
\varphi_n(x)\tau_n(x)=\varphi_n(x)\tau_n(a)+\varphi_n(x)(\tau_n(x)-\tau_n(a)),
$$
а из оценки $|\tau_n(a)|\leqslant T$ следует бесселевость системы $\{\varphi_n(x)\tau_n(a)\}$, то далее, заменив $\tau_n(x)$ на $\tau_n(x)-\tau_n(a)$, можно считать,
что $\tau_n(a)=0$. Тогда
\begin{gather*}
\sum_{n=1}^N|(f,\varphi_n\tau_n)|^2= \sum_{n=1}^N\left|\int_a^bf(x)\overline{\varphi}_n(x)\int_a^x\overline{\tau}_n'(\xi)\,d\xi\,dx
\cdot\int_a^b\overline{f}(y)\varphi_n(y)\int_a^y\tau_n'(\zeta)\,d\zeta\,dy\right|=
\\=\sum_{n=1}^N\left|\int_a^b\int_a^b\overline{\tau}_n'(\xi)\tau_n'(\zeta)
\left(\int_\xi^bf(x)\overline{\varphi}_n(x)\,dx\int_\zeta^b \overline{f}(y)\varphi_n(y)\,dy\right)d\xi\,d\zeta\right|\leqslant\\
\leqslant\int_a^b\int_a^b\tau(\xi)\tau(\zeta)\sum_{n=1}^N\left|(f\chi_{[\xi,b]},\varphi_n)\right|\left|
(\varphi_n,f\chi_{[\zeta,b]})\right|\,d\xi\,d\zeta\leqslant\\
\leqslant c^2\int_a^b\int_a^b\tau(\xi)\tau(\zeta)\cdot\|f\chi_{[\xi,b]}\|\cdot\|f\chi_{[\zeta,b]}\|\,d\xi\,d\zeta \leqslant
c^2\int_a^b\int_a^b\tau(\xi)\tau(\zeta)\,d\xi\,d\zeta\cdot\|f\|^2.
\end{gather*}
Устремив $N\to\infty$, получаем утверждение леммы.
\end{proof}
\begin{theorem}\label{tm.Rieszbas}
Для любого сильно регулярного оператора $\mathcal L_{P,U}$ с потенциалом $P\in L_1[0,\pi]$ вида \eqref{P} система $\{\mathbf y_n\}_{n\in\mathbb Z}$ собственных и
присоединенных функций, введенная в определении \ref{def:EAF}, образует базис Рисса в пространстве $\mathbb H$.
\end{theorem}
\begin{proof}
Воспользуемся теоремой \ref{tm:bari}. Полнота и минимальность системы $\{\mathbf y_n\}_{n\in\mathbb Z}$ уже доказана в теореме \ref{tm:compl}, так что остается
проверить бесселевость систем $\{\mathbf y_n\}_{n\in\mathbb Z}$ и $\{\mathbf w_n\}_{n\in\mathbb Z}$. Вначале мы докажем бесселевость системы $\{\mathbf
y_n\}_{n\in\mathbb Z}$. Поскольку краевые условия сильно регулярны, то все собственные значения $\lambda_n$ оператора $\mathcal L_{P,U}$ просты при $|n|>N$ для
некоторого $N$. Тогда система $\{\mathbf y_n\}_{|n|>N}$ состоит только из нормированных собственных функций оператора $\mathcal L_{P,U}$ и мы можем воспользоваться
асимптотическим представлением \eqref{efas2} $y_{1,n}(x)=e^{i\lambda_nx}\tau_{1,n}(x)$, $y_{2,n}(x)=e^{-i\lambda_nx}\tau_{2,n}(x)$. Тогда из леммы \ref{lem:harm} и
леммы \ref{lem:mp} следует бесселевость систем $\{y_{1,n}\}_{n\in\mathbb Z}$ и $\{y_{2,n}\}_{n\in\mathbb Z}$ в пространстве $L_2[0,\pi]$, что и означает бесселевость
системы $\{\mathbf y_n\}_{n\in\mathbb Z}$ в $\mathbb H$. Перейдем к биортогональной системе $\{\mathbf w_n\}_{n\in\mathbb Z}$. Поскольку функции $\mathbf w_n$ при
$|n|>N$ являются собственными функциями сопряженного оператора $\mathcal L_{P^*,U^*}$, то к системе $\{\mathbf w_n/\|\mathbf w_n\|\}_{|n|>N}$ применимы те же
рассуждения, что и к системе $\{\mathbf y_n\}_{|n|>N}$. Для завершения доказательства достаточно вспомнить (лемма \ref{lem.unimin}), что $\|\mathbf w_n\|<C$ $\forall
n\in\mathbb N$ для некоторой константы $C$.
\end{proof}

\section{Базисность Рисса из подпространств.}

Напомним (см. \cite[Гл.~6]{GK}), что система подпространств $\{\mathcal H_n\}_1^\infty$ называется \textit{базисом} в гильбертовом пространстве $H$, если любой вектор
$x\in H$ разлагается единственным образом в виде ряда $x=\sum_{n=1}^\infty x_n$, где $x_n\in\mathcal  H_n$. Базис $\{\mathcal H_n\}_1^\infty$ из подпространств
является \textit{ортогональным}, если $\mathcal H_n\perp\mathcal  H_m$ при $n\ne m$. Система $\{\mathcal H_n\}_1^\infty$ называется \textit{базисом Рисса из
подпространств}, если существует ограниченный и ограниченно обратимый оператор \(A\) такой, что система \(\{A(\mathcal H_n)\}_1^\infty\) является ортогональным
базисом из подпространств в $H$. В случае регулярных, но не сильно регулярных краевых условий, система $\{\mathbf y_n\}_{n\in\mathbb Z}$ собственных и присоединенных
функций оператора $\mathcal L_{P,U}$ уже не обязана образовывать базис Рисса (см., например, \cite{DM3} и \cite{ShVel}). Можно, однако, показать, что в этом случае
всегда имеется базисность Рисса из подпространств, причем все подпространства двумерны. Идея доказательства этого факта содержится  в статье \cite{SavSh14}. Мы
проведем здесь подробное доказательство, следуя классическим работам \cite{Kaz} и \cite{Mar}. Оно опирается на два замечательных факта --- теорему фон~Неймана и
теорему Карлесона (см. \cite[Гл. VII, теорема 2.2 и лемма 5.4]{Gar}). Мы начнем с доказательства следующего полезного утверждения.
\begin{lemma}\label{lem:summa}
Любой (не обязательно сильно) регулярный оператор Дирака $\mathcal L_{P,U}$ с потенциалом вида \eqref{P} представим в виде суммы $\mathcal L_{P,U}=A+V$, ограниченного
в $\mathbb H$ оператора $V$ и неограниченного замкнутого оператора $A$ с плотной областью определения $\mathfrak D(A)\subset\mathbb H$ и компактной резольвентой. При
этом спектр $\sigma(A)$ расположен в некоторой полосе $\Pi_{\alpha}$ и состоит из собственных значений $\{\lambda_n\}_{n\in\mathbb Z}$. Нумерацию этих собственных
значений (с учетом их алгебраической кратности) можно провести так, чтобы выполнялись асимптотические равенства $\lambda_n=n+\varkappa_j+o(1)$ при $|n|\to\infty$, где
$j=0$, если $n$ четно и $j=1$, если $n$ нечетно, причем геометрическая и алгебраическая кратности каждого собственного значения совпадают, т.е. оператор $A$ не имеет
присоединенных функций. Система нормированных собственных функций оператора $A$ образует базис Рисса в $\mathbb H$.
\end{lemma}
\begin{proof}
Рассмотрим оператор $\mathcal L_{P,U}+V_0$, где $V_0:\,(y_1,y_2)\mapsto (\kappa y_1,-\kappa y_2)$. Если краевые условия, задаваемые матрицей $\mathcal U=(C,\,D)$,
сильно регулярны, то положим $\kappa=0$. В противном случае подберем $\kappa$ следующим образом. Поскольку $\mathcal L_{P,U}+V_0$ есть оператор Дирака вида
\eqref{eq:0.lP} с потенциалом
$\left(\begin{smallmatrix}\kappa& p_2(x)\\
p_3(x)&-\kappa\end{smallmatrix}\right)$, то к нему применимо утверждение \ref{lem:0.2}. Из \eqref{sim} следует, что $\gamma=0$, $\widetilde p_2=p_2$, $\widetilde
p_3=p_3$, т.е. $\mathcal L_{P,U}+V_0=W\mathcal L_{P,\widetilde U}W^{-1}$. При этом краевые условия $\widetilde U$ задаются матрицей $\widetilde{\mathcal
U}=(C,\,e^{i\pi\kappa}D)$. По определению \ref{def:2}, краевые условия $\widetilde U$ сильно регулярны, если
$$
(J_{12}+e^{2i\pi\kappa}J_{34})^2+4e^{2i\pi\kappa}J_{14}J_{23}\ne0.
$$
Таким образом, достаточно выбрать $\kappa$ так, чтобы точка $\mu=e^{2i\pi\kappa}$ не являлась нулем функции
$$
\mu^2J^2_{34}+2\mu(J_{12}J_{34}+2J_{14}J_{23})+J_{12}^2,
$$
что возможно, поскольку эта функция не равна нулю тождественно (это следует из регулярности краевых условий $U$). Итак, мы представили оператор $\mathcal L_{P,U}$ в
виде
$$
\mathcal L_{P,U}=W\mathcal L_{P,\widetilde U}W^{-1}-V_0,
$$
где краевые условия $\widetilde U$ сильно регулярны. Тогда лишь конечное число собственных значений
оператора $\mathcal L_{P,\widetilde U}$ могут иметь алгебраическую кратность, большую единицы. Пусть $\lambda_0$
--- одно из таких собственных значений, $\mathcal H_0$ --- соответствующее корневое подпространство, $\{\mathbf y^{j,k}\}$
--- произвольная каноническая система собственных и присоединенных функций в $\mathcal H_0$, построенная в определении \ref{def:EAF}, а $\mathcal P_0$
--- спектральный проектор на $H_0$ (см. определение \ref{def:EAF}). Определим оператор $K_0$ на подпространстве $\mathcal H_0$ равенствами
$$
K_0\mathbf y^{j,0}=0,\qquad K_0\mathbf y^{j,k}=-\mathbf y^{j,k-1}
$$
для всех $k\ne0$. Тогда оператор $\mathcal L_{P,\widetilde U}+K_0\mathcal P_0$ диагонален на подпространстве $\mathcal H_0$. Пусть оператор $K$ равен сумме операторов
$K_0\mathcal P_0$ по всем корневым подпространствам, отвечающим кратным собственным значениям оператора $\mathcal L_{P,\widetilde U}$ (число слагаемых в этой сумме
конечно, так что оператор $K$ ограничен). Операторы $\mathcal L_{P,\widetilde U}+K$ и $\mathcal L_{P,\widetilde U}$ имеют одинаковый спектр и одинаковую систему
собственных и присоединенных функций, причем все эти функции являются для оператора $\mathcal L_{P,\widetilde U}+K$ собственными. Остается положить
$$
A:=W(\mathcal L_{P,\widetilde U}+K)W^{-1}\qquad \text{и}\qquad V:=-WKW^{-1}-V_0.
$$
\end{proof}
Следующее утверждение известно в теории пространств Харди. Мы, однако, затрудняемся дать точную ссылку и потому приведем его с полным доказательством. Напомним, что
\textit{пространством Харди} $H_\infty$ в верхней полуплоскости $\mathbb C_+=\{z:\mathrm{Im\,} z>0\}$ называется пространство голоморфных и ограниченных в $\mathbb
C_+$ функций с нормой $\|f\|_\infty=\sup_{z\in\mathbb C_+}|f(z)|$.
\begin{proposition}\label{prop:interp}
Пусть последовательность $\{z_n\}_{n\in\mathbb Z}$ лежит в полосе $1\leqslant\mathrm{Im\,} z\leqslant2h$, причем
$$
z_{2n}=2n+\varkappa+o(1)\qquad \text{и} \qquad z_{2n+1}=2n+\varkappa+o(1)
$$
при $|n|\to\infty$. Тогда существует такой номер $N\in\mathbb N$ и такое число $\mu$, что для всякого конечного подмножества $J\subset\{n\in\mathbb Z:|n|\geqslant
N\}$ и всякого номера $K\geqslant\max\{|n|:n\in J\}$ найдется рациональная функция $f_K\in H_\infty$, $\|f_K\|_\infty\leqslant\mu$, такая, что при всех $n$,
$N\leqslant|n|\leqslant K$,
\begin{equation}\label{maininterp}
f_K(z_{2n})=f_K(z_{2n+1})=\begin{cases}1,\quad\text{если }n\in J,\\0,\quad\text{если }n\notin J, \end{cases}\ \text{а если } z_{2n}=z_{2n+1},\ \text{то}\
f_K'(z_{2n})=0.
\end{equation}
\end{proposition}
\begin{definition} (см. \cite[Гл. VII]{Gar})
Последовательность точек $\{z_j\}_{j\in\mathbb N}$ из $\mathbb C_+$ называется {\it интерполяционной}, если
\begin{equation}\label{intseq}
\sup_{n\in\mathbb N}\prod\limits_{k\ne n}\frac{|z_k-z_n|}{|z_k-\overline{z}_n|}\geqslant \delta>0.
\end{equation}
\end{definition}

\begin{THEOREM}[Л.~Карлесон]\label{Carleson2} Пусть $\{z_j\}_{j\in\mathbb N}$
--- интерполяционная последовательность точек верхней полуплоскости. Тогда для любого $K\in\mathbb N$  существуют рациональные функции $f_{j,K}\in
H_\infty$, $1\leqslant j\leqslant K$, такие, что
$$
f_{j,K}(z_l)=\delta_{j\,l},\ 1\leqslant j,\,l\leqslant K,
$$
причем
$$
\sup_{z\in\mathbb C_+}\sum_{j=1}^K|f_{j,K}(z)|\leqslant M,\quad \text{где}\ \ M=\frac{9(3-\delta^2)^2}{4\delta^4}.
$$
\end{THEOREM}

\begin{lemma}\label{lem:interp}
Пусть точки $z_1$ и $z_2$, $z_2\ne z_1$, лежат в полосе $\{z: 1\leqslant\mathrm{Im\,} z\leqslant2h\}$, а числа $w_1$ и $w_2$ произвольны. Тогда существует такая
рациональная функция $\varphi\in H_\infty$, что $\varphi(z_1)=w_1$, $\varphi(z_2)=w_2$, причем
\begin{equation}\label{correction1}
\|\varphi\|_\infty\leqslant8h\left|\frac{w_2-w_1}{z_2-z_1}\right|+2|w_1|+2|w_2|.
\end{equation}
Пусть точка $z_0$  лежит в той же полосе, а $w_0\in\mathbb C$ произвольно. Тогда существует такая рациональная функция $\varphi\in H_\infty$, что $\varphi(z_0)=1$,
$\varphi'(z_0)=w_0$, причем
\begin{equation}\label{correction2}
\|\varphi\|_\infty\leqslant1+4h|w_0|.
\end{equation}
\end{lemma}
\begin{proof}
В первом случае возьмем $\varphi(z)=k\tfrac{z-z_0}{z-\overline{z}_0}$, где числа $z_0\in\mathbb C_+$ и $k\in\mathbb C$ находятся из условий $f(z_1)=w_1$,
$f(z_2)=w_2$. Во втором случае положим $\varphi(z)=1+k\tfrac{z-z_0}{z-\overline{z}_0}$, где число $k$ находится из условия $\varphi'(z_0)=w_0$. Легко видеть, что в
первом случае $\|\varphi\|_\infty=|k|$, а во втором случае $\|\varphi\|_\infty=|k|+1$. Оценки \eqref{correction1} и \eqref{correction2} получаются теперь прямыми
вычислениями, которые мы здесь опускаем.
\end{proof}
\begin{lemma}\label{lem:deltam}
Если точки $\{z_n\}_{n\in\mathbb N}$ лежат в полосе $\{z:1\leqslant\mathrm{Im\,} z\leqslant 2h\}$ и $\inf_{n\ne k}|z_n-z_k|\geqslant1$, то из условия
$$
\sup_{n\in\mathbb N}\sum_{k\ne n}\frac{\mathrm{Im\,} z_n\cdot\mathrm{Im\,} z_k}{|\overline{z}_n-z_k|^2}\leqslant m<\infty
$$
следует \eqref{intseq} с $\delta=e^{-32mh}$.
\end{lemma}
\begin{proof}
Зафиксируем номер $n$ и обозначим $p_k=|z_k-z_n|/|z_k-\overline{z}_n|$. Заметим, что
$$
\frac1{1+4h}\leqslant p_k\leqslant 1,\quad \text{откуда}\ \  -\ln p_k\leqslant8h(1-p_k^2)
$$
Тогда
$$
-\ln\prod\limits_{k\ne n}\frac{|z_k-z_n|}{|z_k-\overline{z}_n|}\leqslant8h\sum_{k\ne n}(1-p_k^2)=32h\sum_{k\ne n}\frac{\mathrm{Im\,} z_k\cdot\mathrm{Im\,}
z_n}{|z_k-\overline{z}_n|^2}\leqslant 32mh.
$$
Отсюда
$$
\sup_{n\in\mathbb N}\prod_{k\ne n}p_k\geqslant e^{-32mh}.
$$
\end{proof}
\noindent\textit{Доказательство утверждения \ref{prop:interp}.} Найдем номер $N_0$, такой, что $|z_{2n}-(2n+\varkappa)|<1/8$ и $|z_{2n+1}-(2n+\varkappa)|<1/8$ при
всех $|n|>N_0$. Отсюда следует, что
$$
|z_{2n}-\overline{z}_{2k}|\geqslant2|k-n|-1,\quad \ \ |n|>N_0,\ |k|>N_0,
$$
а значит
$$
\sum_{n\ne k,\,|n|\geqslant N_0}\frac{\mathrm{Im\,} z_n\cdot\mathrm{Im\,} z_k}{|z_n-\overline{z}_k|^2}\leqslant4h^2\sum_{n\ne k,\,|n|\geqslant
N_0}\frac1{(2|n-k|-1)^2}\leqslant8h^2\sum_{l=1}^\infty\frac1{(2l-1)^2}=\pi^2h^2.
$$
Та же оценка справедлива и для последовательности $\{z_{2n+1}\}_{|n|> N_0}$. Положим $m=\pi^2h^2$. Из леммы \ref{lem:deltam} следует, что обе последовательности
являются интерполяционными, причем число $\delta$ из оценки \eqref{intseq} можно взять равным $e^{-32\pi^2h^3}$. Положим $M=9(3e^{64\pi^2h^3}-1)^2/4$ (см. теорему
\ref{Carleson2}) и найдем номер $N\geqslant N_0$ такой, что
$$
M|z_{2n}-z_{2n+1}|<1/4\qquad \text{для всех}\ \  |n|>N.
$$
Пусть $J$ --- произвольное конечное подмножество $\{n:|n|\geqslant N\}$, а $K\geqslant\max\{|n|:n\in J\}$
--- произвольный номер. Обозначим $\{g_{j,K}\}_{|j|=N}^K$ и $\{h_{j,K}\}_{|j|=N}^K$
--- рациональные функции из теоремы \ref{Carleson2}, построенные по последовательностям $\{z_{2j}\}_{|j|=N}^K$ и $\{z_{2j+1}\}_{|j|=N}^K$ соответственно, т.е.
$$
g_{j,K}(z_{2l})=h_{j,K}(z_{2l+1})=\delta_{jl}.
$$
Далее, для каждого $j$, $N\leqslant |j|\leqslant K$, построим, пользуясь леммой \ref{lem:interp}, функцию $\varphi_{j,K}$ следующим образом. Если $z_{2j}\ne
z_{2j+1}$, то потребуем, чтобы
$$
\varphi_{j,K}(z_{2j})=w_1=\frac1{h_{j,K}(z_{2j})},\qquad \varphi_{j,K}(z_{2j+1})=w_2=\frac1{g_{j,K}(z_{2j+1})}.
$$
Заметим, что для любой функции $f\in H_\infty$ из интегральной формулы Коши следует оценка $\sup_{\mathrm{Im\,} z\geqslant1}|f'(z)|\leqslant\|f\|_\infty$. Поскольку
$$
|h_{j,K}(z_{2j})-1|=|h_{j,K}(z_{2j})-h_{j,K}(z_{2j+1})|\leqslant\sup_{\mathrm{Im\,} z\geqslant1}|h'_{j,K}(z)||z_{2j}-z_{2j+1}|\leqslant
M|z_{2j}-z_{2j+1}|\leqslant1/4,
$$
то числа $|1-w_1|$ и $|1-w_2|$ не превосходят $\min\{4/3 M|z_{2j}-z_{2j+1}|,\,1/3\}$. Тогда из \eqref{correction1} следует, что
$\|\varphi_{j,K}\|_\infty\leqslant24hM+6$. Если $z_{2j}=z_{2j+1}$, то потребуем
$$
\varphi_{j,K}(z_{2j})=1, \qquad\varphi'_{j,K}(z_{2j})=-(g_{j,K}h_{j,K})'(z_{2j})=-g'_{j,K}(z_{2j})-h'_{j,K}(z_{2j}).
$$
Тогда из \eqref{correction2} следует, что $\|\varphi_{j,K}\|_\infty\leqslant8hM+1$. Таким образом, для каждого $j$, $N\leqslant |j|\leqslant K$, определена
рациональная функция
$$
f_{j,K}(z):=g_{j,K}(z)h_{j,K}(z)\varphi_{j,K}(z)\in H_\infty,
$$
для которой 
$$
f_{j,K}(z_{2l})=f_{j,K}(z_{2l+1})=\delta_{j\,l}, \quad  l\ne j,\ N\leqslant |j|,\ l|\leqslant K,
$$
причем
$$
f_{j,K}'(z_{2l})=0,\ \text{если}\ z_{2l}=z_{2l+1}.
$$
Искомую функцию $f_K(z)$ определим суммой $f_K(z)=\sum_{j\in J}f_{j,K}(z)$. Тогда равенства \eqref{maininterp} выполнены и $\forall z\in\mathbb C_+$
$$
|f_K(z)|\leqslant\sum_{j\in J}|f_{j,K}(z)|\leqslant(24hM+6)\sum_{|j|=N}^K|g_{j,K}(z)||h_{j,K}(z)|\leqslant(24hM+6)M^2=\mu.\qquad\square
$$
Приступим к доказательству базисности Рисса из подпространств. Вначале мы сформулируем две теоремы, которые будут использоваться в доказательстве: теорему Гельфанда
(см. \cite[Гл.VI, \S5]{GK}) и теорему фон~Неймана (см. \cite[Гл. XI]{RN}).

\begin{THEOREM}[И.~М.~Гельфанд]\label{Gelfand}
Система $\{\mathcal H_n=\mathrm{Rn\,}\mathcal P_n\}$ является базисом Рисса из подпространств в замыкании своей линейной оболочки тогда и только тогда, когда
\begin{equation}\label{Jprop}
\sup_J\|\sum_{n\in J}\mathcal P_n\|<\infty,
\end{equation}
где супремум берется по всем конечным подмножествам индексов.
\end{THEOREM}

\begin{THEOREM}[Дж.~фон~Нейман]\label{Neuman}
Пусть $T$
--- произвольное сжатие в гильбертовом пространстве, т.е.
$\|T\|\leqslant1$, а функция $f$ голоморфна в круге $|z|<r$, $r>1$, и ограничена в круге $|z|\leqslant1$ константой $\mu$. Тогда $\|f(T)\|\leqslant\mu$.
\end{THEOREM}

Пусть оператор $\mathcal L_{P,U}$ регулярен, но не сильно регулярен. Пусть $\mathcal H_n=\mathrm{Rn\,}\mathcal  P_n$, $|n|\geqslant N_0$, --- корневые подпространства
этого оператора, введенные в определении \ref{def:RP}. Определим дополнительно подпространство $\mathcal H_0=\mathrm{Rn\,}\mathcal  S_{N_0}$, где $\mathcal
S_{N_0}:=1/(2\pi i)\int_\gamma \mathfrak R(\lambda)\,d\lambda$, а замкнутый кусочно--гладкий жорданов контур $\gamma$ охватывает все собственные значения $\lambda_n$
оператора $\mathcal L_{P,U}$ с номерами $n:|n|<2N_0$, и только их.
\begin{theorem} Система $\{\mathcal H_0,\,\mathcal H_n\}_{|n|\geqslant N_0}$ образует базис Рисса из подпространств в пространстве $\mathbb H$.
\end{theorem}
\begin{proof}
Применим теорему \ref{Gelfand}. Из теоремы  \ref{tm:compl} следует, что замыкание линейной оболочки системы $\{\mathcal H_0,\,\mathcal H_n\}_{|n|\geqslant N_0}$
совпадает со всем пространством $\mathbb H$, так что остается доказать выполнение свойства \eqref{Jprop}. Пользуясь леммой \ref{lem:summa}, представим оператор
$\mathcal L_{P,U}$ в виде суммы $\mathcal L_{P,U}=A+V$. Поскольку система собственных функций оператора $A$ образует базис Рисса в пространстве $\mathbb H$, то
найдется такое скалярное произведение $\langle\cdot,\cdot\rangle_1$, топологически эквивалентное исходному (т.е. $c_1\|\cdot\|_1\leqslant\|\cdot\|\leqslant
c_2\|\cdot\|_1$ для некоторых $c_1$ и $c_2$), относительно которого эта система является ортонормированным базисом (см. \cite[Гл.VI, \S2]{GK}). В силу оценки на
нормы, свойство базисности системы подпространств $\{\mathcal H_n\}$ не изменится при переходе к новому скалярному произведению. В новом скалярном произведении
оператор $A$ диагонален в ортогонормированном базисе из своих собственных векторов, т.е. нормален. Тогда числовой образ $\{\langle A\mathbf f,\mathbf f\rangle_1:\
\|\mathbf f\|_1=1\}$ оператора $A$ равен замыканию выпуклой оболочки спектра $\sigma(A)$, а значит лежит в некоторой горизонтальной полосе. Следовательно, числовой
образ оператора $\mathcal L_{P,U}$ (относительно нового скалярного произведения) также лежит в некоторой полосе $\Pi_\alpha$. Поскольку сдвиг не меняет свойств
базисности, то далее можно работать с операторм $B=\mathcal L_{P,U}+i(\alpha+1)$, числовой образ и спектр которого лежат в полосе $1\leqslant\mathrm{Im\,}
z\leqslant2h$, где $h=\alpha+1$. Точки $\{\lambda_n+i(\alpha+1)\}_{n\in\mathbb Z}$ спектра оператора $B$ удовлетворяют условиям утверждения \ref{prop:interp}. Пусть
числа $N$ и $\mu$ определены в формулировке этого утверждения (они зависят только от оператора $\mathcal L_{P,U}$ и, по построению, $N\geqslant N_0$), а
$$
\nu:=\|\mathcal S_{N_0}\|_1+\sum_{|n|=N_0}^{N-1}\|\mathcal P_n\|_1.
$$
Пусть $J\subset\{n\in\mathbb Z:|n|\geqslant N\}$
--- произвольное конечное подмножество, $K$ --- произвольный номер, такой, что $K>\max\{|n|:n\in J\}$, а $f_K$ --- рациональная функция,
построенная в утверждении. Из общей теории функционального исчисления операторов (см., например, \cite[Гл.IX, \S151]{RN}) и представления \eqref{biortog} следует, что
$f_K(B)=\sum_{n\in J}\mathcal P_n$. Пусть $T:=(B-i)(B+i)^{-1}$
--- преобразование Кэли оператора $B$. Легко видеть, что
$$
\forall x\in\mathfrak D(B):\ \|(B+i)x\|_1^2-\|(B-i)x\|_1^2=4\mathrm{Im\,}(Bx,x)_1>0,
$$
откуда
$$
\|(B-i)(B+i)^{-1}x\|_1\leqslant\|x\|_1.
$$
Так как подпространство $\mathfrak D(B)=\mathfrak D(\mathcal L_{P,U})$ плотно в $\mathbb H$, то оператор $T$ продолжается на все пространство и $\|T\|_1\leqslant1$.
Обозначим $g_K(z)=f_K\left(\frac{iz+i}{1-z}\right)$. Тогда, согласно теореме \ref{Neuman}, $\|g_K(T)\|_1\leqslant\mu$. Далее, $\forall x\in H_n$, $|n|>N$, выполнено
$g_K(T)x=f_K(B)x$ при $K\geqslant n$. Переходя к пределу при $K\to\infty$, получим, что $\|\sum_{n\in J}\mathcal P_n\|_1\leqslant\mu\|x\|_1$ на подпространстве
$\overline{\cup_{|n|\geqslant N}H_n}$. Тогда для произвольного $x\in\mathbb H$:
$$
\Bigg\|\sum_{n\in J}\mathcal P_nx\Bigg\|_1=\Bigg\|\sum_{n\in J}\mathcal P_n\Bigg(x-\Bigg(\sum_{|k|=N_0}^{N-1}\mathcal P_k+\mathcal S_{N_0}\Bigg)x\Bigg)\Bigg\|_1
\leqslant\mu(1+\nu)\|x\|_1.
$$
Если теперь $J\subset\{n\in\mathbb Z:|n|\geqslant N_0\}\cup\{0\}$ --- конечное подмножество, то нормы $\|\sum_{n\in J}\mathcal P_n\|_1$ ограничены числом
$\mu+\nu+\mu\nu$, не зависящим от $J$.
\end{proof}

\end{document}